\documentclass{amsart}

\usepackage{amsmath}
\usepackage{amsfonts}
\usepackage{amssymb,enumerate}
\usepackage{amsthm}
\usepackage{cite}
\usepackage{comment}
\usepackage[all]{xy}
\usepackage{hyperref}
\usepackage{mathdots}

\usepackage[font=footnotesize, labelfont={sf,bf}, margin=1cm]{caption}

\theoremstyle{plain}
\newtheorem{lem}{Lemma}[section]
\newtheorem{cor}[lem]{Corollary}

\newtheorem{thm}[lem]{Theorem}

\theoremstyle{definition}
\newtheorem{defn}[lem]{Definition}
\newtheorem{ex}[lem]{Example}

\newtheorem*{convention*}{Convention}

\newcommand{\ds}{\displaystyle}

\newcommand{\bbn}{\mathbb{N}}
\newcommand{\bbz}{\mathbb{Z}}

\newcommand{\bbq}{\mathbb{Q}}
\newcommand{\bbr}{\mathbb{R}}

\newcommand{\onto}{\twoheadrightarrow}

\newcommand{\p}{\mathfrak{p}}

\newcommand{\m}{\mathfrak{m}}

\numberwithin{equation}{section}
\usepackage{color}

\begin{document}
\author{Jim Coykendall}
\address{Department of Mathematical Sciences\\
	Clemson University\\
	Clemson, SC, 29634}
\email[J.~Coykendall]{jcoyken@clemson.edu}
\author{Brandon Goodell}
\address{Department of Mathematical Sciences\\
	Clemson University\\
	Clemson, SC, 29634}
\email[B.~Goodell]{bggoode@clemson.edu}

\title{A Homological Approach to Factorization}
\begin{abstract}
In \cite{mott1974convex}, Mott noted a one-to-one correspondence between saturated and multiplicatively closed subsets of a domain $D$ and directed convex subgroups of the group of divisibility $G(D)$. We construct a functor between inclusions into saturated localizations of $D$ and projections onto partially ordered quotient groups of $G(D)$. We use this functor to construct cochain complexes of $o$-homomorphisms of po-groups. These complexes naturally lead to structure theorems and cohomological results that provide insight into the factorization behavior of $D$.
\end{abstract}

\maketitle

\section{Introduction}\label{introduction}

Classical factorization in integral domains is historically motivated by decomposing objects as products of irreducible elements, which leads suggestively to the study of ordered groups.  Recently, especially since \cite{AAZ}, the main body of literature on factorization has turned its attention toward atomic domains and their associated groups of divisibility. This is natural, of course, since atomic domains are precisely the integral domains in which every non-zero non-unit factors into irreducible elements.  Considering factorization only from the point of view of the atomic elements greatly limits the scope of what can be accomplished. The constraint that integral domains be atomic, generated by their irreducible elements, seems to be just as extreme a constraint as the constraint that integral domains be \textit{antimatter}. Antimatter domains, first described in \cite{coykendall1999integral}, are domains with no irreducible elements whatsoever.  If one were to study a general integral domain, one might reasonably expect that some domain elements decompose into a product of irreducible elements, and some do not. 

In fact, many previous studies have avoided making assumptions of atomicity. In \cite{zaks1976half}, Zaks provides the first discussion on half factorial domains and initially avoids the assumption of atomicity. In \cite{krull1932allgemeine}, Krull also avoids the assumption of atomicity. In \cite{mott1974convex}, Mott considers localizations of arbitrary (not necessarily atomic) integral domains, and in \cite{anderson1990weakly}, Anderson and Zafrullah characterize weakly factorial domains (which are not necessarily atomic).

Examples of rings that are neither atomic nor antimatter are abundant. Recall from \cite{kaplansky1970commutative}, a \textit{valuation domain} is an integral domain $V$ with a field of fractions, $K$, such that if $x \in K \setminus 0$ then $x \in V$ or $x^{-1} \in V$. From \cite{gilmer1972multiplicative}, a discrete valuation domain is a valuation domain in which every primary ideal is a power of its radical. Although a Noetherian discrete valuation domain is a PID, arbitrary valuation domains are not Noetherian in general. A discrete valuation domain $V$ with Krull dimension $2$ and prime spectrum $0 \subseteq \p \subseteq \m$ is not atomic but contains an irreducible.

$V$ admits a unique irreducible (up to associates), say $x$, so the domain is not antimatter. Also, $x$ principally generates the unique maximal ideal, so this domain has a sort of weak universal factorization property: every non-zero non-unit element is divisible by a power of $x$. On the other hand, this domain is not atomic. In fact, no element of the prime ideal with height one, $\p$ factors into a finite product of irreducible elements. Yet $V_\p$ is a PID. In particular, some non-zero non-units in $V$ are products of irreducibles and are not; we refine this example in Section \ref{POAGs}.

We abandon the assumption of atomicity and consider arbitrary integral domains, including antimatter domains. Previous studies in factorization, such as \cite{krull1932allgemeine}, \cite{sheldon1973two}, and \cite{mott1974convex}, have, speaking broadly, focused on sequences of morphisms between algebraic objects and their partial orders. In \cite{krull1932allgemeine}, Krull demonstrated a one-to-one correspondence between the prime ideals of arbitrary valuation rings and the convex subgroups of the associated groups of divisibility.  In \cite{sheldon1973two}, Sheldon extended Krull's work by demonstrating a one-to-one correspondence between the prime ideals of a B\'{e}zout domain, $D$, and the prime filters of the positive cone of the group of divisibility, $G(D)$.  In \cite{mott1974convex}, Mott was the first to notice that prime filters of a positive cone correspond to subgroups that are both convex and directed. Mott used this observation to generalize the correspondences developed by Krull and Sheldon: there exists a one-to-one correspondence between convex and directed subgroups within $G(D)$ and the saturated multiplicatively closed sets within $D$. Mott's generalization seems quite natural, given the connection between convex subgroups and order preserving epimorphisms. In \cite{mott1976exact}, Mott extended his work into exact sequences of value groups.

As discussed in \cite{fuchs2011partially} and \cite{mockor1983groups}, convex subgroups are the only subgroups for which the resulting quotient group is also a po-group under the inherited quotient order. We use Mott's correspondence to develop some homological tools to investigate chains of convex subgroups in arbitrary groups of divisibility, allowing us to qualify the depth of non-atomicity in a domain. We use Mott's correspondence to construct a functor between domain localizations and po-group projections of groups of divisibility. In this way, factorization questions about localizations of $D$ reduce to questions about the structure of the quotient po-groups of $G(D)$. Chains of injections between localizations of $D$ at saturated sets correspond to projections between quotient po-groups of $G(D)$. These sequences provide structure theorems, induce a menagerie of cochain complexes of order-preserving homomorphisms, and yield homological information.

In Section \ref{defsEtc}, we establish terminology, definitions, and properties of partially ordered abelian groups. In Section \ref{POAGs}, we work with Mott's correspondence to construct a functor between chains of o-epimorphisms onto convex directed subgroups of po-groups and chains of localizations of domains at saturated multiplicatively closed subsets. We also construct the fundamental object of this study, which is a sequence of canonical o-epimorphisms which we refer to as the \textit{quasi-atomic quotient sequence}. In Section \ref{homology}, we investigate the homological properties of the quasi-atomic quotient sequence. In Section \ref{structure}, we establish some structure theorems for po-groups under certain assumptions, including a splitting theorem and conditions for when a po-group splits into either the lexicographic or product order.

%%%%%%%%%%%%%%%%%%%%%%%%%%%%%%%%%%%%%%%%%%%
\section{Background and Notation}\label{defsEtc}
%%%%%%%%%%%%%%%%%%%%%%%%%%%%%%%%%%%%%%%%%%%

Let $D$ be an arbitrary integral domain. We refer to an irreducible element of $D$ as an \textit{atom} and finite products of irreducible elements as \textit{atomic}. If every non-zero non-unit element of $D$ is atomic, we refer to $D$ as an \textit{atomic domain}. Following terminology from \cite{coykendall1999integral}, some integral domains may lack irreducible elements. Whenever $D$ contains no irreducible elements, we say $D$ is an \textit{antimatter domain}.  Denote the unit group of $D$ as $U(D)$. 

The study of factorization leads to the study of groups of divisibility, which are abelian and admit a partial order compatible with the group operation. We assume all groups are abelian. Denote the partial order on a set, $X$, as $\leq_X$, unless there is no danger of confusion about which set is being ordered, in which case we simply use $\leq$. A group, $(G,+)$, together with a partial order, $\leq_G$, such that, for any $x, y, z \in G$, if $x \leq_G y$ then $x + z \leq_G y + z$, is known as a \textit{po-group}. The group of divisibility, defined below, is a po-group associated to a domain and has a partial order induced by the divisibility relation in $D$. We denote the group of divisibility of an integral domain, $D$, as $G(D)$.

We use terminology from \cite{fuchs2011partially}: we say an element $g \in G$ is \textit{non-negative} if $e_G \leq_G g$, we say $g$ is \textit{positive} if $e_G <_G g$, and we refer to the subset of all non-negative elements of $G$ as the \textit{positive cone} of $G$, denoted as $G^+$. If $G$ is generated by $G^+$, we say $G$ is \textit{directed}. We use similar terminology for subgroups, i.e.\ we may say a subgroup is directed if it is generated by the positive cone of that subgroup.

Following  Fuchs in \cite{fuchs2011partially} and  Mo\v{c}ko\v{r} in \cite{mockor1983groups}, we refer to any subgroup that is both directed and convex as an \textit{o-ideal}.  If an element $g \in G$ is minimal and positive, we say $g$ is an \textit{atom}. The notion of partial ordering allows for the notion of convexity. We say a subgroup $H \subseteq G$ is \textit{convex} whenever $h_1 \leq g \leq h_2$ and $h_1, h_2 \in H$ implies $g \in H$. If $H$ is directed it is sufficient to check whether $e_G \leq g \leq h$ and $h \in H$ implies $g \in H$. Later, in Section \ref{POAGs}, we demonstrate the connection between convexity in subgroups and the notion of saturation in multiplicatively closed subsets in the domain.

For any po-groups, $G_1, G_2$, and group homomorphism $\phi: G_1 \to G_2$, we say that $\phi$ is \textit{order-preserving} or \textit{monotonic} if $x \leq y$ in $G_1$ implies $\phi(x) \leq \phi(y)$ in $G_2$. We refer to order-preserving (monotonic) homomorphisms as \textit{o-homomorphisms}, following Fuchs in \cite{fuchs2011partially} and  Mo\v{c}ko\v{r} in \cite{mockor1983groups}. Furthermore, if $\phi: G_1 \to G_2$ is a surjective o-homomorphism and $\phi(G_1^+) = G_2^+$, then we say that $\phi$ is an \textit{o-epimorphism}.

This is a stronger condition than simply being a surjective o-homomorphism. 
For example, let $G = (\bbz,+)$ under the usual order, and let $H = G \oplus G$ under the product order. 
Then the map $\phi: H \to G$ defined by  $\phi(a,b) \mapsto 5a + 7b$ is surjective since $\phi(3,-2) = 1$. 
But $\phi$ is also order-preserving: if $(a,b) \geq (0,0)$ then $a \geq 0$ and $b \geq 0$ since $H$ is in the product order, and so $5a + 7b \geq 0$ in $G$. 
However, $1$ cannot be written as the image of a positive element, and so $\phi$ is not an o-epimorphism.  For any o-ideal $H_2 \subseteq G_2$ and any o-epimorphism $\phi:G_1 \to G_2$, the inverse image $\phi^{-1}(H_2) = \left\{g \in G_1 \mid \phi(g) \in H_2\right\}$ is an o-ideal of $G_1$. This is not true of o-homomorphisms in general.

If $\phi$ is a group isomorphism and an o-epimorphism, then both $\phi$ and $\phi^{-1}$ are  \textit{o-isomorphisms}.  We are generally only concerned with partially ordered groups up to o-isomorphism; if two po-groups are o-isomorphic, we will consider them to be the ``same.'' Indeed if rings $R_1$ and $R_2$ have o-isomorphic groups of divisibility $G_1 \simeq G_2$, then any factorization structure detectable by the group of divisibility in $R_1$ will be present in  $R_2$ up to units via the o-isomorphism on $G_1, G_2$.

The group of divisibility is defined in \cite{mott1974convex}, \cite{mockor1983groups}, and \cite{jensen1963characterizations}, as the po-group of non-zero principal fractional ideals. This group is partially ordered by reverse set containment, and the group is, of course, abelian. For an integral domain $D$ with quotient field $K$ and unit group $U(D)$, the group of divisibility is defined in \cite{gilmer1972multiplicative} and \cite{anderson1990weakly} as $K^{\times} = K\setminus 0$. The ordering on $K^{\times}/U(D)$ is the \textit{natural divisibility ordering} defined by $aU(D) \leq bU(D)$ if and only if $\frac{b}{a} \in D$. There exists an o-isomorphism between the group of divisibility $G(D)$ so defined and the group $K^{\times}/U(D)$ defined by mapping the principal fractional ideal generated by $\frac{a}{b}$ to the group element $\frac{a}{b}U(D)$.   We use the latter definition from \cite{gilmer1972multiplicative} and \cite{anderson1990weakly} by defining $G(D) = K^{\times}/U(D)$.

There exists a natural semi-valuation from $\nu : K \setminus 0 \longrightarrow G(D)$ given by the map $\nu: x \mapsto xU(D)$. This natural semi-valuation connects ring-theoretic factorization to po-group-theoretic information. For example, $D$ is an atomic domain if and only if $G(D)$ is generated by its atoms, and $D$ is an antimatter domain if and only if $G(D)$ contains no minimal positive elements. By definition, $G(D)^+ = \nu(D)$.  Also, any $x \in D$ is an atom (as a ring element) if and only if $\nu(x) \in G(D)$ is an atom (as a po-group element).

A partial order on direct sums of po-groups can be naturally induced from the underlying partial orders; for two examples, consider the lexicographic order and the product order. For po-groups $G_1, G_2$, the lexicographic order, $\leq_\ell$, on $G_1 \oplus G_2$ is defined by saying $(g_1, g_2) \leq_\ell (h_1, h_2)$ if and only if $g_1 < h_1$ in $G_1$ or $g_1 = h_1$ in $G_1$ and $g_2 \leq h_2$ in $G_2$. On the other hand, the product order, $\leq_p$ on $G_1 \oplus G_2$ is defined by saying $(g_1, g_2) \leq_p (h_1, h_2)$ if and only if $g_1 \leq h_1$ in $G_1$ and $g_2 \leq h_2$ in $G_2$. If $G_1, G_2$ are totally ordered then the lexicographic order on $G_1 \oplus G_2$ is a total ordering. These are not the only two product partial orders available. In this way, the product order $\leq_p$ is finer than the lexicographic order, $\leq_\ell$, and the equality relation is the finest of all relations. These definitions extend inductively to any finite direct sum $\oplus_{i=1}^{N} G_i$. 

We can further extend the definition of the product order to any set of po-groups $\left\{G_i\right\}_{i \in \Lambda}$ whose index set $\Lambda$. The product order on $\oplus_{i \in \Lambda} G_i$ is defined by setting $(g_i) \leq_p (h_i)$ if and only if $g_i \leq h_i$ in $G_i$ for every $i \in \Lambda$. We can similarly extend the definition of the lexicographic order to any set of po-groups whose index set, $\Lambda$, is partially ordered. The lexicographic order on $\oplus_{i \in \Lambda} G_i$ where $\Lambda$ is partially ordered by $\leq_{\Lambda}$ is defined by setting $(g_i) \leq_\ell (h_i)$ if and only if there exists some $\lambda_0 \in \Lambda$ such that $g_{\lambda_{0}} < h_{\lambda_{0}}$ in $G_{\lambda_0}$ and if $i < \lambda_0$ then $g_i = h_i$. The lexicographic order $\leq_\ell$ on $\oplus_{i\in \Lambda} G_i$ so defined is a partial order. If each $G_i$ is totally ordered and $\Lambda$ is a well-ordered set then the lexicographic order on $\oplus_{i \in \Lambda} G_i$ is a total ordering.

Following \cite{samuel1948ultrafilters}, a set $S$ with two relations, say $\sim$ and $\approx$ such that  $s \sim t$ implies $s \approx t$, we say $\sim$ is \textit{finer} than $\approx$ (or, alternatively, $\approx$ is \textit{coarser} than $\sim$). Placing the coarse/fine relation on partial orders becomes more interesting when we turn our attention to direct sums of po-groups. In fact, the product order is finer than than the lexicographic order, and the discrete partial order (the equality relation) is finer than the product order.

%%%%%%%%%%%%%%%%%%%%%%%%%%%%%%%%%%%%%%%%%%%
\section{Partially Ordered Abelian Groups}\label{POAGs}
%%%%%%%%%%%%%%%%%%%%%%%%%%%%%%%%%%%%%%%%%%%

In this section, we construct the primary object of study, the quasi-atomic quotient sequence. To do so, we study convex subgroups of po-groups, quotient po-group projections, and isomorphism theorems in the po-group setting (even isomorphism theorems do not necessarily hold for po-groups). We move on to Mott's correspondence between convex directed subgroups of the group of divisibility and saturated and multiplicatively closed sets of the underlying integral domains. We develop a functorial relationship between localizations of rings and projections of groups. From this, we construct chains of projections of groups related to ring localizations, which we call the quasi-atomic quotient sequence. We observe a few surprising connections between an integral domain and the associated quasi-atomic quotient sequence for the group of divisibility. 

Let $G$ be a multiplicative po-group with order $\leq_G$ (although not necessarily a group of divisibility). Let $H \subseteq G$ be a subgroup.  The \textit{induced quotient order} $\leq_{G/H}$ on $G/H$, defined naturally by
\begin{align*}
aH &\leq_{G/H} bH\text{ if and only if }\exists \alpha \in aH, \beta \in bH\text{ such that } \alpha \leq_{G} \beta
\end{align*}
is a \textit{quasi-order} (some authors prefer \textit{pre-order}), which are relations that satisfy all the axioms of a partial ordering except for antisymmetry.  Equivalently, we may say $aH \leq bH$ if and only if there exists some $h$ such that $a \leq bh$. For certain choices of $H$, however, the quotient order is a partial order. In fact, we have Theorem \ref{thm-tfae}, in part due to Fuchs in \cite{fuchs2011partially}, and certainly many others, which describes when the quotient order is a partial order.  

In Theorem \ref{thm-tfae} below, we present a usual result regarding convex subgroups and an extension of that result relating to saturation in a ring.  Recall that, for a domain $D$ with quotient field $K$, for any $x \in K$,  $xU(D)\in G(D)$ is positive if and only if $x \in D$. From this perspective, property (iv) relates to saturation.

\begin{thm}\label{thm-tfae}
Let $G$ be a multiplicative po-group with subgroup $H \subseteq G$ with partial order $\leq_G$.  The following are equivalent:
\begin{enumerate}[(i)]
\item $H$ is a convex subgroup of $G$,
\item $G/H$ is a po-group under the quotient order $\leq_{G/H}$,
\item $H$ is the kernel of some o-epimorphism, and
\item For any $g_1, g_2 \in G^+$, if $g_1 g_2 \in H$, then $g_1 \in H$ and $g_2 \in H$.
\end{enumerate}
\end{thm}
\begin{proof}

Denote the identity of $G$ as $e_G$. We only prove the equivalency of (i) and (iv); Fuchs proved the equivalency of (i), (ii), and (iii) in \cite{fuchs2011partially}. If $g_1, g_2 \in G^+$ then certainly $g_1 g_2 \in G^+$. If we further have that $H$ is convex and that $g_1 g_2 \in H$ then $e_{G} \leq g_1 \leq g_1 g_2$ and $e_{G} \leq g_2 \leq g_1 g_2$. Convexity provides (iv). On the other hand, assume $H$ satisfies (iv) and $h_1 \leq g \leq h_2$ for some $h_1, h_2 \in H$.  Then $h_1^{-1}g$ and $g^{-1}h_2$ are both positive elements whose product is in $H$. Since $H$ satisfies (iv), we obtain that both $h_1^{-1}g$ and $g^{-1}h_2$ are in $H$, yielding that $g \in H$.
\end{proof}

Unfortunately, Theorem \ref{thm-tfae} is as restrictive as it is descriptive: convex subgroups are the only subgroups for whom the natural quotient group is partially ordered under the inherited quotient order. Of course, not all subgroups of an arbitrary po-group $G$ are convex.  For example, if $H$ is any infinite proper subgroup of the additive $\bbz$ then $\bbz/H$ is finite, so only the trivial partial order on $\bbz/H$ is compatible with addition. Thus, $\bbz/H$ is not a po-group and $H$ is not a convex subgroup of $\bbz$.

We also have Theorem \ref{thm-first-o-iso}, also due to Fuchs in \cite{fuchs2011partially} (and others). This theorem is the analogue to the First Isomorphism Theorem in the po-group setting. The proof of Theorem \ref{thm-first-o-iso} may be found in \cite{fuchs2011partially}, so we omit it:

\begin{thm}[First o-isomorphism theorem]
If there exists an o-epimorphism $\phi:G_1 \to G_2$ then there exists an o-isomorphism $G_2 \simeq G_1/\ker \phi$.
\label{thm-first-o-iso}
\end{thm}

Theorem \ref{thm-first-o-iso} may be used to establish the following corollary regarding quotients and direct sums:

\begin{cor}
Let $\Lambda$ be an index set, let $\{G_i\}_{i\in\Lambda}$ be a set of po-groups with corresponding o-ideals $\{H_i\}_{i\in\Lambda}$. Define the po-groups $G_p:=\oplus_{i\in\Lambda} G_i$,  $H_p:=\oplus_{i \in \Lambda} H_i$, and $L_p := \oplus_{i\in \Lambda} G_i/H_i$, all under the product order. There exists an o-isomorphism $G_p/H_p \simeq L_p$. 
\label{direct-sums-respect-lex-prod}
\end{cor}
\begin{proof}
The canonical map $\pi: G_p \to L_p$ defined by mapping $(g_i)_i \mapsto (g_i + H_i)_i$ is a surjective o-homomorphism. The kernel of this map is $H_p$, and a direct sum of convex subgroups under the product order is a convex subgroup, so $H_p$ is convex. Hence, $H_p$ is the kernel of the o-epimorphism $G_p \onto G_p/H_p$. Moreover, by Theorem \ref{thm-first-o-iso}, $G_p/H_p \simeq \text{Im}(\pi) \subseteq L_p$. Note that if $x \in L_p$, then $x = (g_i + H_i)$ for some $\left\{g_i\right\}_{i \in \Lambda}$. Of course, $x = \pi((g_i))$, and so we see $L_p \subseteq \text{Im}(\pi)$.
\end{proof}

The lexicographic analogue of this Corollary is false in general. To see this, consider po-groups $G_1, G_2$ with subgroups $H_1, H_2$. Let $(h_1, g_1)$, $(h_2, g_2) \in G_1 \oplus G_2$ ordered lexicographically such that $h_1 < h_2$ in $H_1 \subseteq G_1$ but $g_2 < g_1$ in $G_2 \setminus H_2$. The canonical map is then $\pi: G_1 \oplus G_2 \onto \frac{G_1}{H_1} \oplus \frac{G_2}{H_2}$. Furthermore, we have that $(h_1, g_1) \leq (h_2, g_2)$ in $G_1 \oplus G_2$ ordered lexicographically, but $\pi(h_1, g_1) = (H_1, g_1 + H_2)$ and $\pi(h_2, g_2) = (H_1, g_2 + H_2)$. Since $g_2 < g_1$, we have that $g_2 + H_2 < g_1 + H_2$, so $\pi(h_2, g_2) < \pi(h_1, g_1)$. This violates order preservation.

Recall we defined an o-ideal of a po-group to be any subgroup that is simultaneously convex and directed. We produce an example demonstrating that directed subgroups and convex subgroups are distinct in general.

\begin{ex}
Not all directed subgroups are convex. For an example of such a subgroup, consider the additive subgroup $2\bbz \subseteq \bbz$ under the usual total ordering. Also, not all convex subgroups are directed. For an example of such a subgroup, consider again the integers and their group of divisibility, $G(\bbz) = \bbq^{\times}/U(\bbz)$ under the usual partial order induced by ordinary divisibility. We claim the subgroup $H:=\langle 2/3 \rangle$ is convex. Indeed, assume $(\frac{2}{3})^{n} \leq x \leq (\frac{2}{3})^{m}$ for some $n, m \in \bbz$. In particular, $(\frac{2}{3})^{n}\leq(\frac{2}{3})^{m}$. The partial order insists, then, that $\frac{(2/3)^m}{(2/3)^n} \in \bbz$, thus $n=m$. Antisymmetry provides $x = (\frac{2}{3})^n$ and so $H$ is convex but not directed.
\label{convexNotDirectedAndViceVersa}
$\triangle$\end{ex}

\begin{thm}[\protect{Mott's Correspondence~\cite[ThmX]{mott1974convex}}]\label{thm-mott}
Let $D$ be an integral domain with quotient field $K$, unit group $U(D)$, and group of divisibility $G(D) = K^{\times}/U(D)$. Let $\nu:K^{\times} \to G(D)$ be the natural map defined by $x \mapsto xU(D)$. Let $\mathcal{S}$ be the set of all saturated multiplicatively closed subsets of $D$ and let $\mathcal{O}$ be the set of all o-ideals of $G(D)$. Then the map from $\mathcal{S}$ to $\mathcal{O}$ given by $S \mapsto \langle \nu(S)\rangle$ is a one-to-one correspondence. Further, $G(D)/\langle \nu(S)\rangle$ is precisely the group of divisibility of $D_S$.
\end{thm}

Mott's correspondence suggests the existence of a functor connecting the environments relating canonical inclusions between localizations of $D$ to o-epimorphisms between quotient po-groups of $G(D)$. 

\begin{defn}\label{categs}
Fix an integral domain $D$ with quotient field $K$.
\begin{enumerate}[(a)]

\item \label{categ-r} Let $\Re$ be all localizations of $D$ at saturated multiplicatively closed subsets. For saturated multiplicatively closed subsets, $S$, $T$, define the morphisms
\[\text{Hom}(D_S,D_T)=\left\{\epsilon_{S,T}:D_{S} \to D_{T} \mid \epsilon_{S,T}\left(r\right)= r/1\right\}\] whenever $S \subseteq T$ and $\emptyset$ otherwise. Then $\Re$ is a category.

\item \label{categ-g} Let $\mathfrak{G}$ be all quotient groups of $G(D)$ via o-ideals. For o-ideals $H$ and $L$ define the morphisms \[\text{Hom}\left(G/H, G/L\right)= \left\{\pi_{H,L}:G/H \to G/L \vert \pi_{H,L}\left(gH\right) = gL\right\}\] whenever $H \subseteq L$ and $\emptyset$ otherwise. Then $\mathfrak{G}$ is a category.

\end{enumerate}
\end{defn}

It is clear that $\Re$ has an initial object ($D$), a terminal object ($K$), $\mathfrak{G}$ has an initial object ($G(D)$) and $\mathfrak{G}$ has a terminal object (the trivial group). We flesh out Mott's Correspondence Theorem by describing a functor between these categories. Recall that a functor is \textit{essentially surjective} if it is surjective on objects.

\begin{thm}
Fix an integral domain $D$. Let $\digamma: \Re \to \mathfrak{G}$ be the functor defined by
\begin{align*}
\digamma(D_S)& =G(D)/\langle\nu(S)\rangle\\
S \subseteq T \Rightarrow \digamma(\epsilon_{S,T}) &=\pi_{\langle \nu(S)\rangle,\langle \nu(T)\rangle}
\end{align*}
then $\digamma$ defines an covariant functor from the category $\Re$ to the category $\mathfrak{G}$. Furthermore, Mott's correspondence implies this functor is essentially surjective.
\begin{figure}[h!]
\centerline{
\xymatrix{
D_S \ar[r]^-{\epsilon_{S,T}} \ar[d]^-{\digamma} & D_T \ar[d]^-{\digamma} \\ G(D_S) \ar[r]^-{\pi_{\langle \nu(S),\nu(T)\rangle}} & G(D_T)
}}
\caption{The commutative diagram implied by the group of divisibility functor. Injections between saturated localizations correspond to projections between groups of divisibility.}
\end{figure}

\label{thm-funct}
\end{thm}
\begin{proof}
We have that $\digamma(\epsilon_{S,S}) = \pi_{\langle \nu(S)\rangle,\langle \nu(S)\rangle}$ is defined as the identity function. That is, $\digamma(\epsilon_{S,S}) = \text{id}_{G(D)/\langle \nu(S)\rangle}= \text{id}_{\digamma(D_S)}$. Also, for the containment $S \subseteq T \subseteq U$, then certainly $\epsilon_{S,U} = \epsilon_{T,U} \circ \epsilon_{S,T}$. Hence, we have that $\digamma(\epsilon_{S,U}) = \digamma(\epsilon_{T,U} \circ \epsilon_{S,T})$.  Of course, $\digamma(\epsilon_{S,U}) = \pi_{\langle\nu(S)\rangle,\langle\nu(U)\rangle}$, and we obtain
\begin{align*}
\digamma(\epsilon_{S,U}) = \pi_{\langle\nu(S)\rangle,\langle\nu(U)\rangle} &= \pi_{\langle\nu(T)\rangle,\langle\nu(U)\rangle} \circ \pi_{\langle\nu(S)\rangle,\langle\nu(T)\rangle}\\
&= \digamma(\epsilon_{T,U}) \circ \digamma(\epsilon_{(S,T)})
\end{align*}

Finally, for any $g\langle \nu(S)\rangle \in G(D)/\langle \nu(S)\rangle$, we have
\begin{align*}
\pi_{\langle \nu(T)\rangle,\langle \nu(U)\rangle} \circ \pi_{\langle \nu(S)\rangle,\langle \nu(T)\rangle}(g\langle \nu(S)\rangle) &= \pi_{\langle \nu(T)\rangle,\langle \nu(U)\rangle}(g\langle \nu(T)\rangle)\\
&= g\langle \nu(U)\rangle\\
&= \pi_{\langle \nu(S)\rangle,\langle \nu(U)\rangle}\left(g\langle \nu(S)\rangle\right)
\end{align*}

Hence $\digamma$ is a covariant functor. All that remains is to establish essential surjectivity, but this is guaranteed by Mott's correspondence: any po-group of the form $G(D)/H$ where $H$ is an o-ideal has an associated saturated multiplicatively closed subset $S \subseteq D$ such that $G(D)/H = G(D_S)$.
\end{proof}

With Theorem \ref{thm-funct} we translate between the group-theoretic setting and the ring-theoretic setting: a localization of $D$ may be compatibly associated with a quotient po-group of $G(D)$. The only question remaining is the choice of (saturated) multiplicatively closed sets we use in localization. One ostensibly obvious choice is the set of non-atomic elements. Below in Example \ref{exDontLocalizeAtNonAtomic} we demonstrate this set may not be multiplicatively closed. 

\begin{ex}
Let $x$ be an indeterminate over $\mathbb{F}_2$, let $\mathcal{X}$ be the set defined as $\mathcal{X} := \left\{x\right\}\cup\left\{x^{2/3^n}\mid n \geq 1\right\}$ and let $\m \subseteq \mathbb{F}_2[\mathcal{X}]$ be the maximal ideal generated by all monomials. Let $R := \left(\mathbb{F}_2\left[\mathcal{X}\right]\right)_{\m}$ be the localization at $\m$. Note since $R$ is the localization of $\mathbb{F}_2\left[\mathcal{X}\right]$ at the maximal ideal generated by all monomials, any term with a nonzero constant term (such as $1 + x^{2/3}$) is a unit. 

We claim that $x$ is irreducible and $x^{2/3}$ is not atomic, so the factorization
\[x\cdot x = x^{2/3} \cdot x^{2/3} \cdot x^{2/3}\]
demonstrates that the set of non-atomic elements in a domain is not necessarily multiplicatively closed. We also show elements of the form $x + x^{2/3}$ are irreducible (a similar argument can be used to determine that an element of the form $x + f(x)$ where $x \nmid f(x)$ is irreducible), so that the equation
\begin{align}\label{eqn:two-thirds-relation}
x^{2/3}(x+x^{2/3})(x+x^{2/3}) =& x^2(1+x^{2/3})
\end{align}
shows that in the field of fractions for $R$, we have
\[x^{2/3} = \frac{ux^2}{\xi^2}\]
where the expression on the right is a ratio of atoms. This provides an example of a non-atomic element, namely $x^{2/3}U(R)$, which is in the subgroup generated by the atoms of the group of divisibility.

To prove $x$ is irreducible, consider the additive monoid generated over $\bbn$ by $\left\{ 1, \frac{2}{3}, \frac{2}{9},\cdots,\frac{2}{3^n},\cdots\right\}$. We call this monoid the \textit{monoid of exponents}, which we shall denote $M$. Our terminology may be obvious, as the multiplicative monoid generated by all (associate classes of) monomials in $R$ is isomorphic to the additive monoid of exponents. Hence, establishing that $x$ is irreducible is equivalent to establishing $1$ is irreducible in $M$. We go further and establish that $x$ is the uniquely irreducible monomial (up to associates) by demonstrating that $1$ is uniquely irreducible in $M$.

Since $1 \in M$, it is also true that $\bbn \subseteq M$. Assume $1$ is reducible. Then we may write
\[
1=n_1\frac{2}{3}+\cdots +n_t\frac{2}{3^t}
\]
where we can choose, without loss of generality, $0\leq n_i\leq 2$ for all $1\leq i\leq t$ and $n_t\neq 0$. Multiplying both sides of this equation by $3^t$ we obtain
\[
3^t=2n_1(3^{t-1})+2n_2(3^{t-2})+\cdots +2n_{t-1}(3)+2n_t.
\]
Reducing this equation modulo $3$ gives that $2n_t \equiv 0$. Hence, $n_t$ is divisible by $3$, contradicting our previous choice. We conclude that $1$ is irreducible in $M$, i.e.\ $x$ is irreducible in $R$. 

Moreover, any other generator of $M$ decomposes as $\frac{2}{3^n} = \frac{2}{3^{n+1}}+\frac{2}{3^{n+1}}+\frac{2}{3^{n+1}}$; we conclude $1$ is uniquely irreducible in $M$, and hence is $x$ is uniquely irreducible among monomials in $R$. Hence, all atomic monomials are associate to some $x^n$.

We prove each element of the form $x + x^{2/3^n}$ (or, formally, $\frac{x+x^{2/3^n}}{1}$) is irreducible. Assume non-zero non-units $\alpha, \beta \in R$  satisfy $\alpha\beta = x + x^{2/3^n} \in R$. $R$ is a localization so there exists some polynomials $f(x), g(x) \in \mathbb{F}_2[\mathcal{X}]$ and $h(x), k(x) \notin \m \subseteq \mathbb{F}_2[\mathcal{X}]$ such that $\alpha = \frac{f(x)}{h(x)}$, $\beta = \frac{g(x)}{k(x)}$. Moreover, since $\alpha \beta = \frac{x+x^{2/3^n}}{1} \in R$, we have that $f(x) g(x) = (x+x^{2/3^n})h(x)k(x)$ in 
$\mathbb{F}_2[\mathcal{X}]$. 

Since $h, k \notin \m \subseteq \mathbb{F}_2[\mathcal{X}]$, $h(0) = k(0)=1$. Write $f = a_0 + \sum_{i=1}^{N_f} a_i x^{n_i} \in \mathbb{F}_2[\mathcal{X}]$ and $g = b_0 + \sum_{i=1}^{N_g} b_i x^{m_i} \in \mathbb{F}_2[\mathcal{X}]$ for some $N_f, N_g$, and each $0 \neq n_i, m_i \in M$.  Since $\alpha, \beta$ are assumed to be non-zero non-units, we have that $a_0 = b_0 = 0$. 
\begin{align*}
    \underbrace{(\sum_{i=1}^{N_f} a_i x^{n_i})}_{f(x)}\underbrace{(\sum_{i=1}^{N_g} b_i x^{m_i})}_{g(x)} =& (x+x^{2/3^n})\underbrace{(1 + \cdots)}_{h(x)}\underbrace{(1+\cdots)}_{k(x)}
\end{align*}
In particular, we have some $n_i + m_j = 1$, which is irreducible in $M$, so $n_i = 0$ or $m_j = 0$, contradicting our choice of $f$ and $g$ as non-units. Hence, $x+x^{2/3^n}$ is irreducible. 
\label{exDontLocalizeAtNonAtomic}
$\triangle$\end{ex}

This example is more complicated than necessary to demonstrate the set of non-atomic elements is not multiplicatively closed, but we use it again for other results later. Example \ref{exDontLocalizeAtNonAtomic} presents a valuable opportunity to expand on the notion of the atomic subgroup. This inspires the following definition, first set forth in \cite{boynton2013graph}:

\begin{defn}
Given an integral domain $D$, we say $x \in D$ is an \textit{almost-atomic domain element} if there exists some atomic element $a \in D$ such that $xa$ is an atomic element. We refer to any domain in which all elements are almost-atomic elements as an \textit{almost-atomic domain}. 
\label{def:almost-atomic}
\end{defn}

We produce an example of an integral domain in which the non-atomic elements form a multiplicatively closed set, but localization at this set is still not favorable.

\begin{ex}\label{first-pathological-example}
Let $R$ be the discrete valuation domain from the introduction with Krull dimension $2$ and prime spectrum $(0)\subset \p \subset \m $. The maximal ideal, $\m$, is principal; we select a generator, $x \in \m$, and write $\m = (x)$. Any element of $\m \setminus\p$ can be uniquely (up to associates) factored into a power of $x$. Although $R$ is not atomic, any proper overring of $R$ is atomic. Any non-zero non-unit element of $R$ that is not a power of $x$ is also not atomic. The set of non-unit, non-atomic elements (i.e.\ $\p \setminus 0$) is certainly multiplicatively closed, but its saturation is $R\setminus 0$. The corresponding localization, $R_S$ with $S = R\setminus 0$, is the quotient field of $R$; every element in the localization is a unit.
$\triangle$\end{ex}

Examples \ref{exDontLocalizeAtNonAtomic} and \ref{first-pathological-example} discourage choosing the set of non-atomic elements while using Theorems \ref{thm-mott} and \ref{thm-funct}; perhaps we ought to localize at sets generated by atomic elements. Our strategy is now to turn our attention to the saturated multiplicatively closed subsets of $D$ generated by the atomic elements, or equivalently, by citing Theorem \ref{thm-mott}, we turn our attention to subgroups of $G(D)$ generated by its atoms. To this end, we make the following definitions:

\begin{defn}
Similar to Definition \ref{def:almost-atomic}, we say $x$ is a \textit{quasi-atomic domain element} if there exists some $y \in D$ such that $xy$ is atomic. We refer to any integral domain in which all elements are quasi-atomic as a \textit{quasi-atomic domain}.
\label{def:atomicAndquasiAtomic:one}
\end{defn}

Certainly every almost-atomic element (domain) is a quasi-atomic element (domain, respectively). It is natural to demand an example of a quasi-atomic element or domain that is not almost-atomic to discern the difference between these concepts. Construction of examples of quasi-atomic domains that are not almost-atomic is a delicate task. In \cite{lebowitz2016classifying}, it is shown that the domain $D = \bbz[x] + x^2\bbr[x]$ is quasi-atomic but not almost-atomic.

We take the ring-theoretic concepts of atomicity and quasi-atomicity and extend them to be po-group-theoretic:

\begin{defn}Given a po-group, $G$, we define the \textit{atomic subgroup}, $A(G)\subseteq G$, to be the subgroup of $G$ generated by the atoms (minimal positive elements) of $G$.  We say $g \in G$ is a \textit{quasi-atomic group element} if $g \in G^+$ and if there exists some $h \in G^+$ such that $gh \in A(G)$. We refer to the subgroup of $G$ generated by the set of all quasi-atomic group elements, $Q(G) \subseteq G$, as the \textit{quasi-atomic subgroup}.
\label{def:atomicAndquasiAtomic:two}
\end{defn}

Notice that there is an obvious way of defining an almost-atomic subgroup, but this subgroup coincides with $A(G)$. Also notice that, in general, the atomic subgroup, $A(G)$, is not an o-ideal; in fact, in Theorem \ref{thm-sat}, we prove that the quasi-atomic subgroup, $Q(G)$, is the smallest o-ideal containing $A(G)$. 

These subgroups similarly detect information about factorization in $D$. For the group of divisibility, if $D$ is an atomic domain, then $G(D) = A(G(D))$, and if $G(D) = A(G(D))$ then $D$ almost-atomic. Since an antimatter domain contains no irreducibles, there are no quasi-atomic elements. Hence, antimatter domains have groups of divisibility with $Q(G(D)) = A(G(D)) = \left\{e_{G(D)}\right\}$.  For an arbitrary po-group, $G$, which may or may not be a group of divisibility, if $A(G) = G$. Then $G$ is directed because, in this case, $G$ is generated by the minimal positive elements. 

Every atomic element (domain) is almost-atomic, and every almost-atomic element (domain) is quasi-atomic. We have $\left\{e_G\right\} \subseteq A(G) \subseteq Q(G) \subseteq G$ and these are strict in general. For an example of strict containment in $Q(G) \subset G$,  take any antimatter domain $D$ that is not a field. Then $Q(G(D))$ is trivial since $D$ has no atoms, but $G(D)$ is nontrivial since $D$ is not a field. For an example of strict containment in $\left\{e_G\right\} \subset A(G) \subset Q(G)$, recall the domain $D = \bbz[x] + x^2\bbr[x]$ from \cite{lebowitz2016classifying} is quasi-atomic but not almost-atomic. Hence we have that $A(G(D)) \subset Q(G(D))$ and we have the proper containment $\left\{e_G\right\} \subset A(G(D))$ since $D$ contains irreducibles. Lemma \ref{thm-struct-final} shows that the quasi-atomic subgroup and almost-atomic subgroup do not detect the difference between direct sums and direct products.

\begin{lem}
Let $\left\{G_{\alpha}\mid \alpha \in \Lambda \right\}$ be a family of po-groups. Define $G = \oplus_{\alpha} G_{\alpha}$ and $H = \prod_{\alpha} G_{\alpha}$, both in the product order.  Then $A(G) = A(H) = \oplus_{\alpha} A(G_{\alpha})$ and $Q(G) = Q(H) = \oplus_{\alpha} Q(G_{\alpha})$.
\label{thm-struct-final}
\end{lem}
\begin{proof}
Observe that in both the direct product and the direct sum, the minimal positive elements, if they exist, are sequences of the form $m_\beta =\{\epsilon_\alpha\}$ where
\[
\epsilon_\alpha=
\begin{cases}
0, \text{ if } \alpha \neq \beta \\
p, \text{ if }\alpha =\beta
\end{cases}
\]
where $p$ is an atom in $G_\alpha$. Thus the generating sets for $A(G)$ and $A(H)$ are identical. Further, we claim that $Q(H)^+ \subseteq Q(G)^+$. To show this, it is sufficient to show that an element of $Q(H)^+$ has only a finite collection of non-zero coordinates (despite that $Q(H)^+ \subseteq H$, which is the direct product). Let $(x_\alpha) \in Q(H)^+$. Since $H$ has the product order, each $x_\alpha$ is non-negative in $H$. Since $(x_\alpha) \in Q(H)^+$, there exists some $(h_\alpha) \in H^+$ such that $(x_\alpha + h_\alpha) \in A(H)$. We claim that both $(x_\alpha)$ and $(h_\alpha) \in Q(G)^+$, namely that they both have a finite collection of non-zero coordinates. We have shown that $A(H) = A(G)$ and we have $(x_\alpha + h_\alpha) \in A(G)$ so only a finite collection of $x_\alpha + h_\alpha$ are nonzero. Since $x_\alpha$ and $h_\alpha$ are both non-negative, their sum is zero if and only if both $x_\alpha$ and $h_\alpha$ are zero. We conclude each of $(x_\alpha)$ and $(h_\alpha)$ have only a finite collection of non-zero coordinates. Thus, $(h_\alpha) \in G^+$ and so $(x_\alpha) \in Q(G)^+$.  This establishes that $Q(H)^+ \subseteq Q(G)^+$.

On the other hand, since $G \subseteq H$, we have $G^+ \subseteq H^+$. If $(x_\alpha) \in Q(G)^+$ then there exists some $(g_\alpha) \in G^+$ such that $(x_\alpha + g_\alpha) \in A(G) = A(H)$, so $(x_\alpha) \in Q(H)^+$.
\end{proof}

Unfortunately, Lemma \ref{thm-struct-final} does not extend to the lexicographically ordered direct sums or products. The generators $\left\{g \in G^+ \mid \exists h \in G^+, gh \in A(G)\right\}$ for $Q(G)$ are a specific case of a more general characterization:

\begin{defn}
Let $G$ be a multiplicative po-group and consider a chain of directed subgroups $H \subseteq L \subseteq G$. Consider the subgroup 
\[(H:L) := \langle \left\{g \in G^+ \mid \exists \ell \in L^+\text{ such that } \ell  g \in H\right\}\rangle\]
to be the $\leq$-\textit{semi-saturation} of $H$ with respect to $L$.
\label{def:orderSat}
\end{defn}

We use the term semi-saturation to avoid confusion with the notion of a convex subgroup from Theorem \ref{thm-tfae}, which corresponds with the notion of a saturated multiplicatively closed subset of a ring. We examine some of the properties of $(H:L)$. Note that if $G$ is not directed, $(H:G)$ may be the trivial group. Further, if $H, L$ are not directed, we can still define $(H:L)$ as in Definition \ref{def:orderSat}, but then we no longer are guaranteed that $H \subseteq (H:L)$. Lemma \ref{lem:semi-sat} demonstrates that subgroups of the form $(H:L)$ are nicely behaved in the sense that containments are prserved.
\begin{lem}\label{lem:semi-sat}
Let $G$ be a multiplicative po-group with an arbitrary chain of directed subgroups $H \subseteq L \subseteq L^{\prime} \subseteq  \subseteq G$. Then
\begin{enumerate}[(i)]
\item $H \subseteq (H:L)  \subseteq (H:L^{\prime}) \subseteq (H:G)$,
\item $(H:G) \subseteq (L:G) \subseteq (L^{\prime}:G)$, and
\item if $H$ is an o-ideal then $H=(H:G)$.
\end{enumerate}
\end{lem}
\begin{proof}
For (i), since $H$ is directed, each $h \in H^+$ can be written $h e_G$, and $e_G \in L^+$, so $h \in (H:L)$. If $g \in (H:L)^+$ then there exists some $\ell \in L$ such that $g \ell \in H$. But $\ell \in L \subseteq L^{\prime}$ so $g \in (H:L^{\prime})$. In particular, if $L^{\prime} = G$, we have $g \in (H:G)$. Thus, we have the containment $H \subseteq (H:L) \subseteq (H:L^{\prime}) \subseteq G$. For (ii), if $g \in (H:G)^+$, then $gg^{\prime} \in H \subseteq L \subseteq L^{\prime}$ for some $g^{\prime} \in G^+$. Hence, $g \in (L:G)^+$ and $(L^{\prime}:G)$. Likewise, if $g \in (L:G)^+$, then $g \in (L^{\prime}:G)^+$. For (iii), if $H$ is an o-ideal and $g \in G^+$ has some $g^{\prime} \in G^+$ such that $gg^{\prime} \in H$, then $e_G \leq g \leq g g^{\prime} \in H$, so $g \in H$. Thus, if $H$ is an o-ideal, then $(H:G) = H$. Following (i), we have that $H = (H:L)$ for any subgroup in the containment $H \subseteq L \subseteq G$.

\end{proof}

We apply Definitions \ref{def:atomicAndquasiAtomic:two} and \ref{def:orderSat} together with Lemma \ref{lem:semi-sat} to the atomic subgroup $H = A(G)$ of an arbitrary po-group. 
\begin{cor}
Let $G$ be a directed po-group and $A(G)$ the atomic subgroup and $Q(G)$ the quasi-atomic subgroup. Then $Q(G) = (A(G):G)$. Also, $A(G)$ is convex if and only if $A(G) = Q(G)$.
\end{cor}

\begin{thm}
If $G$ is a directed po-group with a directed subgroup, $H$, and $\mathcal{O}_{H}$ is the set of all o-ideals of $G$ containing $H$, then 
\[\left(H:G\right) = \underset{H^{\prime} \in \mathcal{O}_H}{\bigcap} H^{\prime}\]
In particular, $(H:G)$ is the smallest o-ideal containing $H$. 
\label{thm-sat}
\end{thm}
\begin{proof}
We first prove that $(H:G)$ is convex and directed, so $(H:G) \in \mathcal{O}_H$ and $\ds \cap H^{\prime} \subseteq (H:G)$. By construction, $(H:G)$ is generated by its positive elements and is thus directed. To establish convexity, suppose that $e_G \leq y\leq x$ where $x\in (H:G)^+$. We can write $x=pn^{-1}$ where $p, n$ are two positive generators of $(H:G)$ (that is, there are positive elements $s,t$ such that both $ps$ and $nt$ are in $H$). Reworking the previous inequality, we have $n\leq yn\leq p$ and so $p(yn)^{-1}$ and $yns$ are positive elements such that their product is an element of $H$, i.e.\ $ps\in H$. Therefore, we have that $yns\in (H:G)$ and since $n,s\in (H:G)$ as well, we have that $y\in (H:G)$. Thus, $(H:G)$ is convex by Theorem \ref{thm-tfae} and is therefore an o-ideal.  Hence, $\cap H^{\prime} \subseteq (H:G)$.

To establish that $(H:G) \subseteq \cap H^{\prime}$, let $x$ be a positive generator of $(H:G)$. There exists an element $g\in G^+$ such that $xg=h\in H$. Since $H\subseteq H^{\prime}$ and $H^{\prime}$ is convex, we have that $x\in H^{\prime}$. Hence, $(H:G) \subseteq H^{\prime}$ for any $H^{\prime} \in \mathcal{O}_{H}$, so $(H:G) \subseteq \displaystyle \cap H^{\prime}$.
\end{proof}
\begin{cor}
Let $G$ be a directed po-group with atomic subgroup $A(G)$. Then $Q(G)$ is the smallest o-ideal of $G$ containing $A(G)$.
\end{cor}
\begin{cor}
Let $G$ be a directed po-group. For any subgroup $H \subseteq G$, there exists a unique o-epimorphism $\pi_H: G \onto G/(H:G)$ such that if $f: G \to G^{\prime}$ is an o-homomorphism and $H \subseteq \ker(f)$ then $f$ factors through $\pi_H$ in the sense that there exists an o-homomorphism $\hat{f}:G/(H:G) \to G^{\prime}$ such that $f = \hat{f} \circ \pi_H$.
\end{cor}

Note that since $(H:G)$ is the smallest o-ideal of $G$ containing $H$, the containment $H \subseteq (H:L) \subseteq (H:G)$ suggests that $(H:L)$ is not convex in general. In the special case that $G=G(D)$, by Mott's Correspondence (Theorem \ref{thm-mott}), we have that $G/(H:G) = G(D_S)$ for some saturated multiplicatively closed set $S \subseteq D$.  If $A(G(D))$ is not convex then the multiplicatively closed set generated by all atoms in $D$ is not saturated, implying the existence of quasi-atomic elements in $D$ that are not almost-atomic.  To see how convexity is violated in a quasi-atomic group that is not almost-atomic, let $D$ be any integral domain with an element $x \in D$ that is quasi-atomic but not almost-atomic. Since $x$ is quasi-atomic, we have some $y \in D$ such that $xy$ is atomic. Thus, in the group of divisibility, we have that $U(D) \leq x U(D) \leq xy U(D)$ where $U(D) \in A(G(D))$ and $xy \in A(G(D))$. However, $xU(D)$ is not almost-atomic so $xU(D) \notin A(G(D))$.

Due to their connection with the saturated multiplicatively closed subsets generated by the atomic elements of $D$, o-ideals play a distinguished role in factorization. With Theorem \ref{thm-tfae}, we may project onto $G/Q(G)$ providing an o-epimorphism. With Theorem \ref{thm-mott}, the resulting po-group is a group of divisibility for an associated localization. In other words, for an integral domain $D_0:=D$ and ascending chain of saturated multiplicatively closed sets $S_0 \subseteq S_1 \subseteq \ldots$, defining $D_{n+1} := D_{S_{n}}$ for each $n \geq 0$, we can view any chain of canonical monomorphisms between localizations
\[D_0 \subseteq D_1 \subseteq D_2 \subseteq \ldots\]
as a chain of canonical po-group o-epimorphisms
\[G(D_0) \onto G(D_1) \onto G(D_2) \onto \ldots\]
where each $G(D_{n+1}) = G(D_n)/Q(G(D_n))$ by Mott's correspondence. We avoid difficulty with indexing by presuming $S_0 = U(D)$ so $D_0 \cong D$.

\begin{defn}
Let $G_{0}:= G$ be a po-group, $G_{1}:=G_{0}/Q(G_{0})$, and $\pi_0:G_0 \to G_1$ the natural o-epimorphism.  For each $n > 0$, inductively define $G_{n+1}:=G_{n}/Q(G_{n})$ and $\pi_n: G_n \to G_{n+1}$ the natural o-epimorphism. We say the sequence of o-epimorphisms
\[G_0 \onto G_1 \onto G_2 \onto \dots \]
is the \textit{quasi-atomic quotient sequence} of $G$.
\end{defn}

Given a general sequence of composable maps in any category, we say the sequence is \textit{stable at degree} $n$ (or simply \textit{stable}) if there exists a minimal index $n \in \bbn_0$ such that each morphism in degree $n + k \in \bbn_0$ is an isomorphism.  If a sequence is stable, the isomorphism class may be the terminal object in the category, in which case we say the sequence \textit{terminates}, or is \textit{terminally stable}, \textit{terminal}, or \textit{bounded}. Every terminally stable sequence is stable. The quasi-atomic quotient sequence stabilizes to o-isomorphisms if and only if some quotient group is antimatter. To see this, note that antimatter po-groups are precisely the po-groups for which $Q(G) = \left\{e_G\right\}$. On the other hand, the quasi-atomic quotient sequence terminally stabilizes if and only if some quotient group, $G_n$, is quasi-atomic. To see this, simply note that $G_n$ is quasi-atomic implies $G_{n+1}$ is trivial (and hence $G_{n+1}$ is antimatter). 

Below in Example \ref{ex-n-atomic} we provide examples of stable sequences and terminally stable sequences. We provide further examples in Section \ref{structure} from the ring-theoretic points of view.

\begin{ex}\label{ex-n-atomic}
\begin{enumerate}[(i)]
\item We say that a quasi-atomic quotient sequences that terminally stabilizes after a finite number of steps is \textit{bounded}. Let $G = \oplus_{i=1}^{n} \bbz$, denoted $\bbz^n$, and take the order on $G$ to be lexicographic. Then $G$ has a single atom, $(1, 0, \dots, 0)$, and $A(G) = \bbz \oplus 0^{n-1} \simeq \bbz$. Hence, we obtain \[G_1 = \frac{\bbz^{n}}{\bbz \oplus 0^{n-1}} \simeq \bbz^{n-1}\] ordered lexicographically. We obtain the bounded quasi-atomic quotient sequence
\[\bbz^n \to \bbz^{n-1} \to \dots \to \bbz^2 \to \bbz \to 0 \to \dots \]
which terminally stabilizes at degree $n$. $\triangle$

\item Some quasi-atomic quotient sequences non-termainally stabilize at a finite degree. Let $H = \bbz$ be ordered in the product order, and let $G = (\oplus_{i =0}^{n} H) \oplus \bbq$, ordered lexicographically. Note $G$ has a unique atom, $(1,0,0,\cdots,0)$, which generates $\bbz \oplus 0 \oplus 0 \oplus \cdots 0 $. Thus, $A(G) = H \oplus 0 \oplus 0 \oplus \cdots \oplus 0$, and we obtain 
\begin{align*}
    G_1 =&  \frac{(\oplus_{i=0}^{n} H)\oplus \bbq}{H  \oplus 0 \oplus \cdots \oplus 0}  \simeq (\oplus_{i=0}^{n-1} H) \oplus \bbq
\end{align*}
Inductively we obtain
\begin{align*}
    G_2 =& \frac{(\oplus_{i=0}^{n-1} H)\oplus \bbq}{H  \oplus 0 \oplus \cdots \oplus 0} \simeq (\oplus_{i=0}^{n-2} H) \oplus \bbq\\
    \vdots & \\
    G_{n-1} =& \frac{H\oplus H \oplus \bbq}{H \oplus 0\oplus 0 } \simeq H \oplus \bbq\\
    G_n =& \frac{H \oplus \bbq}{H \oplus 0} \simeq \bbq
\end{align*}
And for $m \geq n$, we have $G_m \simeq G_n$. We obtain the quasi-atomic quotient sequence
\[G_0 \to G_1 \to G_2 \to \cdots \to G_{n-1} \to \bbq \to \bbq \to \bbq \to \bbq \to \cdots \]
which is stable at degree $n$. $\triangle$

\item Some quasi-atomic quotient sequences never stabilize. Let $G = G_0 = \oplus_{i \in \bbn} \bbz$ be ordered lexicographically. Then $G$ has a unique atom, $(1,0,0,\ldots)$ generating $Q(G) = \bbz \oplus\left(\oplus_{i \geq 1} 0\right)$. Then $G_1 = G_0/Q(G_0) = \frac{\bbz \oplus \bbz \oplus \bbz \oplus \cdots}{\bbz \oplus 0 \oplus 0 \oplus \cdots} \simeq 0 \oplus \bbz \oplus \bbz \oplus \ldots$. Note $G_1 \simeq G_0$; in fact, we obtain the quasi-atomic quotient sequence 
\[(\oplus_{i \geq 0} \bbz) \to (\oplus_{i \geq 1} \bbz) \to (\oplus_{i \geq 2} \bbz) \to \cdots\] which is o-isomorphic in each degree to the sequence $G_0 \to G_0 \to \cdots$. However, each canonical o-epimorphism $\pi: G_0 \to G_1$ is not an o-isomorphism, so this sequence is not stable. $\triangle$
\end{enumerate}
\end{ex}

Example \ref{ex-n-atomic} demonstrate nontrivial cases of termination and stabilization, and that the finite cases can become arbitrarily long, even for simple examples. To demonstrate that these group-theoretic ideas are obtainable from the ring-theoretic motivations of this study, we provide some examples of integral domains that yield groups of divisibility that exhibit nontrivial termination and stabilization in Section \ref{structure}.

Classifying $G$ by the behavior of its quasi-atomic quotient sequence motivates the following definitions:

\begin{defn}\label{defn-atomic} Given the quasi-atomic quotient sequence of $G$, define the following:
\begin{enumerate}[(a)]
\item when $n\in \bbn_0$ is the least integer such that $G_n = \left\{e\right\}$, we say $G$ is $n$-\textit{atomic}, 
\item when $G$ is not $m$-atomic for any $m \in \bbn_0$ and $n \in \bbn_0$ is the least integer such that for each $k \in \bbn_0$, $\pi_{n+k}:G_{n+k} \to G_{n+k+1}$ is an isomorphism, we say $G$ is $n$-\textit{antimatter}, and
\item when $G$ is not $m$-atomic for any $m \in \bbn_0$ and not $n$-antimatter for any $n \in \bbn_0$, we say $G$ is \textit{mixed-atomic-antimatter}.
\end{enumerate}
\end{defn}

Note that the quasi-atomic quotient sequence may stabilize into a chain of o-isomorphic po-groups and yet the natural surjections, $\pi_n$, are not o-isomorphisms. 

An $n$-atomic po-group $G$ is terminally stable in the $n^{th}$ degree and thus is stable in the $n^{th}$ degree, and so $n$-atomic po-groups are $n$-antimatter. Groups satisfying these definitions provide some immediate properties.  For example, the only $0$-atomic group is the trivial group and the group of divisibility of an antimatter domain is $0$-antimatter. A nontrivial po-group is quasi-atomic if and only if it is $1$-atomic. For any nontrivial group $G$, $Q(G)$ is $1$-atomic as a group unto itself. If $G$ satisfies the ascending chain condition on o-ideals, then $G$ is $n$-atomic or $n$-antimatter for some $n$.  It is immediately clear that Example \ref{ex-n-atomic}(i) is $n$-atomic and Example \ref{ex-n-atomic}(ii) is $n$-antimatter.

We use Lemma \ref{thm-struct-final} to obtain Corollary \ref{structureCorollary}, which ties our notions of $n$-antimatter and $n$-atomic developed in Definition \ref{defn-atomic} in with direct sums in the product order.

\begin{cor}\label{structureCorollary}
Let $\left\{G_{\alpha} \mid \alpha \in \Lambda \right\}$ be a family of po-groups, $G=\oplus_{\alpha} G_{\alpha}$ in the product order, let $\left\{n_{\alpha} \mid \alpha \in \Lambda\right\}$ be a net of natural numbers such that the supremum $N=\sup_{\alpha \in \Lambda}\left\{n_{\alpha}\right\}$ is finite. The following hold:
\begin{enumerate}[(i)]
\item if for every $\alpha \in \Lambda$, $G_{\alpha}$ is $n_{\alpha}$-atomic, then $G$ is $N$-atomic;
\item if for every $\alpha \in \Lambda$, $G_{\alpha}$ is $n_{\alpha}$-antimatter, then $G$ is $N$-antimatter.
\end{enumerate}
\end{cor}
\begin{proof}
We index the following quotient groups:
\begin{align*}
G_0 &:= \oplus_{\alpha} G_{\alpha}\\
G_n &:=G_{n-1}/Q(G_{n-1})\text{ for any }n\geq 1\\
G^{(0)}_{\alpha} &:= G_{\alpha}\\
G^{(n)}_{\alpha} &:= G^{(n-1)}_{\alpha}/Q(G^{(n-1)}_{\alpha})\text{ for any }n\geq 1\\
\end{align*}
By Corollary \ref{direct-sums-respect-lex-prod}, $G_n \simeq \oplus_{\alpha} G^{(n)}_{\alpha}$ for any $n \geq 0$. Since each $G_{\alpha}$ is $n_{\alpha}$-atomic, each $G^{(n_{\alpha})}_{\alpha}$ is trivial.  Hence, we obtain the sequence $G_0 \to G_1 \to G_2 \to \cdots$, which is o-isomorphic to the sequence, in direct sum notation
\[\oplus_{\alpha} G^{(0)}_{\alpha} \to \oplus_{\alpha} G^{(1)}_{\alpha} \to \oplus_{\alpha} G^{(2)}_{\alpha} \to \dots\]
Since each $n_{\alpha} \leq N$, we have that this sequence terminates by the $N^{th}$ step and no earlier.

This establishes the first statement of the corollary. The second statement is proved similarly with all sequences stabilizing to a non-trivial group by the $N^{th}$ step, rather than terminating.
\end{proof}

\begin{ex} Let $R = \bbz$; since $\bbz$ is atomic (and hence $1$-atomic), we have that $G(R) \simeq Q(G(R))$. In fact, $Q(G(R)) \simeq \oplus_{i\in \bbn} \bbz$ under the product order. The o-isomorphism is precisely the map defined by $\pm p_1^{e_1}p_2^{e_2} ... p_n^{e_n} \mapsto \sum_{i=1}^{n}e_i$ in which $p_i$ is the $i^{th}$ prime integer, and each $e_m \in \bbn$ for $1 \leq m \leq n$.
\label{ex-terminates-immediately}
$\triangle$\end{ex}

In analogy to Example \ref{ex-terminates-immediately}, a CK domain $D$ in which every irreducible is prime has a group of divisibility $G(D)=\oplus_{i=1}^{n} \bbz$ under the product order where $n$ is the number of non-associate primes. Of course, any such domain is atomic and therefore has $0$-atomic quasi-atomic quotient sequence.

The above definitions relate to ideas presented by other authors. For example, in \cite{mott1974convex}, Mott developed a dimension theory for po-groups in the following way: for a totally ordered po-group, $G$, with $n$ distinct convex subgroups, $\dim(G) = n$. It is clear that for any totally ordered $G$ with $\dim(G) = n$, we have that the quasi-atomic quotient sequence is $m$-atomic or $m$-antimatter for some $m \leq n$. Following immediately from these definitions, if $V$ is a valuation domain then $G(V)$ is totally ordered, and if $V$ has Krull dimension $n$ then $\text{dim}(G(V))=n$. 

The quasi-atomic quotient sequence and Definition \ref{defn-atomic} cannot distinguish between quasi-atomic, almost-atomic, and atomic domains. The atomic subgroup in the group of divisibility not only contains all atoms but also all almost-atomic elements. Furthermore, the atomic subgroup is contained inside the quasi-atomic subgroup, which is the kernel of the differential. All quasi-atomic domain elements are units after each stage of localization. 

Rather than classifying the ``niceness'' of factorization behavior, these definitions classify the depth of pathological factorization. For a quasi-atomic $D$ we have the quasi-atomic quotient sequence $G(D) \to 0 \to 0 \to \cdots$ which terminally stabilizes in degree $n=1$. Only for integral domains with pathological factorization behavior will our construction yield interesting results.

%%%%%%%%%%%%%%%%%%%%%%%%%%%%%%%%%%%%%%%%%%%%%%%%%%%%%%%%%%%%%%%%%%%%%%%%%%%%%%%%%%%%%%%%%%
\section{Cohomology Theory for Quotient Sequences}\label{homology}
%%%%%%%%%%%%%%%%%%%%%%%%%%%%%%%%%%%%%%%%%%%%%%%%%%%%%%%%%%%%%%%%%%%%%%%%%%%%%%%%%%%%%%%%%%

In this section, we apply cohomological tools to the quasi-atomic quotient sequence. These tools roughly quantify the failure of atomicity within a po-group $G$, and hence may be applied to groups of divisibility to quantify how far an integral domain may be from being atomic. However, we approach this section with an eye toward a general po-group before we turn our attention to groups of divisibility in particular.  We are able to demonstrate that properties within cohomology groups correspond with factorization behavior. For example, the cohomology groups generated from the quasi-atomic quotient sequence need not be partially ordered abelian groups, and so may admit torsion elements. Furthermore, these torsion elements correspond to domain elements with specific factorization behavior.

In this section, we let $G$ be a multiplicative po-group with identity $1$ unless otherwise stated. Let $G_0:=G$ and consider the quasi-atomic quotient sequence
\[G_0 \to G_1 \to G_2 \to \dots\]
wherein $G_{n+1} = G_n/Q(G_n)$. We may construct a more detailed picture of the quasi-atomic quotient sequence including the atomic and quasi-atomic subgroups in the following commutative diagram:
\[\xymatrix{
G_0 \ar[r]^-{\pi_0} & G_1 \ar[r]^-{\pi_1} & G_2 \ar[r]^-{\pi_2} & \dots\\
Q(G_0) \ar[u] \ar[ddr] & Q(G_1) \ar[u] \ar[ddr] & Q(G_2) \ar[u] & \\
A(G_0) \ar[u] & A(G_1) \ar[u] & A(G_2) \ar[u] & \\
1 \ar[u] & 1 \ar[u] & 1 \ar[u] &
}\]
where upward arrows denote the natural inclusion maps and rightward arrows denote the natural o-homomorphism, $\pi_n$, restricted to their appropriate domains. Padding the diagram on the left with the trivial groups in the usual manner to obtain a bi-infinite diagram, we immediately obtain the cochain complexes
\begin{align*}
    A_\bullet =& \cdots \to 1 \to A(G_0) \to A(G_1) \to A(G_2) \to \cdots\\
    Q_\bullet =& \cdots \to 1 \to Q(G_0) \to Q(G_1) \to Q(G_2) \to \cdots\\
    Q_\bullet/A_\bullet =& \cdots \to 1 \to \frac{Q(G_0)}{A(G_0)} \to \frac{Q(G_1)}{A(G_1)} \to \frac{Q(G_2)}{A(G_2)} \to \cdots\\
\end{align*}
For short, we may denote $A(G_n) = A_n$ and $Q(G_n) = Q_n$. Note that all of these complexes are trivial in the sense that each map is the trivial o-homomorphism mapping all elements to the identity.  To construct cochain complexes that do not consist of the trivial o-homomorphisms in each degree, we consider the inverse images of the atomic and quasi-atomic subgroups pulled back through their differentials.

\begin{defn}
Define $\widehat{A_{n}}:=\pi_{n}^{-1}(A_{n+1})$ and $\widehat{Q_{n}}=\pi_{n}^{-1}(Q_{n+1})$.
\end{defn}

These groups allow us to resolve the quasi-atomic quotient sequence in a yet more detailed commutative diagram:
\begin{equation}
\xymatrix{
 \dots \ar[r] & G_{n-1} \ar[r]^-{\pi_{n-1}} & G_n \ar[r]^-{\pi_n} & G_{n+1} \ar[r]^-{\pi_{n+1}}  & \dots \\
\dots \ar[ddr] & \widehat{Q}_{n-1}  \ar[u] \ar[ddr] & \widehat{Q}_{n} \ar[u] \ar[ddr] & \widehat{Q}_{n+1} \ar[ddr] \ar[u] & \\
\dots \ar[ddr] & \widehat{A}_{n-1} \ar[u] \ar[ddr] & \widehat{A}_{n} \ar[u] \ar[ddr] & \widehat{A}_{n+1} \ar[u] \ar[ddr] \\
\dots \ar[ddr] & Q_{n-1} \ar[u] \ar[ddr] & Q_{n} \ar[u] \ar[ddr] & Q_{n+1} \ar[u]  \ar[ddr] & \dots  \\
& A_{n-1} \ar[u] &A_{n} \ar[u] & A_{n+1} \ar[u] & \dots\\
 & 1 \ar[u]& 1 \ar[u]& 1 \ar[u] & \dots
}\label{detailedCommutativeDiagram}
\end{equation}
where, again, each upward arrow denotes the natural inclusion map and each rightward arrow denotes the natural o-epimorphism, $\pi_n$, restricted appropriately. Certainly, the natural inclusion maps are each o-homomorphisms. Commutativity of this diagram is a standard diagram chase.

Recall from Section \ref{defsEtc} that if $H_2 \subseteq G_2$ is an o-ideal and $\phi:G_1 \to G_2$ is an o-epimorphism, then $\phi^{-1}(H_2)$ is an o-ideal of $G_1$.  Due to the containment of $A_{n+1} \subseteq Q_{n+1}$, we have the containment $A_{n} \subseteq Q_{n} \subseteq \widehat{A}_{n} \subseteq \widehat{Q}_{n}$. The subgroups $A_n$, $Q_n$, $\widehat{A}_n$, and $\widehat{Q}_n$ each yield cochain complexes with maps induced by $\{\pi_{n}\vert n\geq 0\}$ restricted appropriately, and we may build a variety of cochain complexes from these. We formalize a few of these with Lemma \ref{lem-cochain-complexes}, and we include a sketch of the proof for brevity.

\begin{lem}
Each of the following are cochain complexes of o-homomorphisms of po-groups:
\begin{align}
A_{\bullet}&:= \dots \longrightarrow 1  \longrightarrow A_{0} \longrightarrow  A_{1} \longrightarrow  A_{2} \longrightarrow  \dots \\
Q_{\bullet}&:= \dots \longrightarrow 1  \longrightarrow Q_{0} \longrightarrow Q_{1} \longrightarrow  Q_{2} \longrightarrow \dots \\
\widehat{A}_{\bullet}&:= \dots  \longrightarrow 1 \longrightarrow \widehat{A}_{0} \longrightarrow \widehat{A}_{1} \longrightarrow \widehat{A}_{2} \longrightarrow \dots \\
\widehat{Q}_{\bullet}&:= \dots \longrightarrow  1 \longrightarrow \widehat{Q}_{0} \longrightarrow \widehat{Q}_{1} \longrightarrow \widehat{Q}_{2} \longrightarrow \dots \\
\widehat{A}_{\bullet}/Q_{\bullet}&:= \dots \longrightarrow  1 \longrightarrow \widehat{A}_{0}/Q_{0} \longrightarrow \widehat{A}_{1}/Q_{1} \longrightarrow \widehat{A}_{2}/Q_{2} \longrightarrow \dots \\
\widehat{Q}_{\bullet}/Q_{\bullet}&:= \dots \longrightarrow  1 \longrightarrow \widehat{Q}_{0}/Q_{0} \longrightarrow \widehat{Q}_{1}/Q_{1} \longrightarrow \widehat{Q}_{2}/Q_{2} \longrightarrow \dots
\end{align}
in which the differentials, $\delta_n$, are naturally induced by the o-epimorphisms $\pi_n$.  Furthermore, the sequences $A_{\bullet}$, $Q_{\bullet}$, $\widehat{A}_{\bullet}/Q_{\bullet}$, $\widehat{Q}_{\bullet}/Q_{\bullet}$ are each trivial in the sense that each differential is the trivial homomorphism. Lastly, the sequence $\widehat{Q}_{\bullet}$ is exact.
\label{lem-cochain-complexes}
\end{lem}
\begin{proof}
We have that the above sequences form cochain complexes by the construction of the groups, verified by an easy diagram chase as previously described. Triviality and exactness follow immediately from the fact that $\ker \pi_{n} = Q_{n}$ and hence $A_n\subseteq \ker \pi_n$.
\end{proof}

Note we might be tempted to consider $Q_\bullet/A_\bullet$, but since each $A_n$ is not convex in general, the result is not a cochain complex of po-groups (although each degree is pre-ordered). The differentials in the cochain complexes of Lemma \ref{lem-cochain-complexes} are o-homomorphisms obtained by restricting the domain po-groups of the canonically determined o-epimorphisms. Furthermore, each $Q_n = \ker(\pi_n)$, so each differential is the zero map. We form short exact sequences of cochain complexes of po-group o-homomorphisms to obtain appropriate cohomology groups in Lemma \ref{three-little-ses} and demonstrate relationships between them in Definition \ref{def-coh}, the proof of which we omit to avoid a diagram chase.

\begin{lem}\label{three-little-ses}
The following sequences of cochain complexes of o-homomorphisms of po-groups are short exact:
\begin{align}
1 \to Q_{\bullet} \to & \widehat{A}_{\bullet} \to \frac{\widehat{A}_{\bullet}}{Q_{\bullet}} \to 1\\
1 \to Q_{\bullet} \to & \widehat{Q}_{\bullet} \to \frac{\widehat{Q}_{\bullet}}{Q_{\bullet}} \to 1
\end{align}
where the chain maps between the complexes are induced by inclusion or projection where appropriate.
\end{lem}

\begin{defn}\label{def-coh}
Let $X_{\bullet}$ denote any of the complexes from Lemma \ref{lem-cochain-complexes}, and denote the differential maps as $\delta_{i}: X_{i}\longrightarrow X_{i+1}$. For $n \in \bbz$, define the $n^{th}$ cohomology group of $X_{\bullet}$ by $H^n(G,X_{\bullet}):=\ker(\delta_{n})/\text{Im}(\delta_{n-1})$. \end{defn}

Definition \ref{def-coh}, together with the cochain complexes in Lemma \ref{lem-cochain-complexes}, yields the following cohomology groups for every $n \geq 1$:
\begin{align}
    \label{eqn1} H^n(G,A_\bullet) =& A_n,\\
    \label{eqn2}H^n(G,Q_\bullet) =& Q_n, \\
    \label{eqn3}H^n(G,\widehat{A}_\bullet) =& \frac{Q_n}{A_n},\\
    \label{eqn4}H^n(G,\widehat{Q}_\bullet) =& 1, \\
    \label{eqn5}H^n(G,\frac{\widehat{A}_\bullet}{Q_\bullet}) =& \frac{\widehat{A}_n}{Q_n} = \pi_n\left(\widehat{A}_n\right) = A_{n+1}, \\
    \label{eqn6} H^n(G,\frac{\widehat{Q}_\bullet}{Q_\bullet}) =& \frac{\widehat{Q}_n}{Q_n} = \pi_n\left(\widehat{Q}_n\right) = Q_{n+1},
\end{align}
and, further, these all hold true for $n=0$ except $H^0(G,\widehat{A}_\bullet) = H^0(G,\widehat{Q}_\bullet) = Q_0$.

We stress, as before, that the differentials of the cochain complexes are not o-epimorphisms in general. The cohomology groups of the form $H^n(G,\widehat{A}_\bullet) = Q_n/A_n$ are related to the gap between almost-atomicity and quasi-atomicity. Since $A_n$ is not convex in general, $H^n(G,\widehat{A}_\bullet)$ is not partially ordered under the inherited quotient order. Further, it may have torsion elements if there exists some $q \in Q_n$ such that $q^m \in A_n$ for some $m \in \bbn$. In the special case that $G$ is the group of divisibility for an integral domain $D$, torsion in a cohomology group corresponds to specific factorization behavior. Arbitrary elements of $H^n(G,\widehat{A}_\bullet)$ are cosets of the atomic subgroup $A_n$ with representatives from $Q_n$, i.e.\ $H^n(G,\widehat{A}_\bullet) = \left\{q + A_n \mid q \in Q_n\right\}$. These are torsion if and only if $q^k \in A_n$ for some $k \in \bbn$. Since these are elements of $G_n = G(D_{S_n})$, we see that torsion elements of $H^n(G,\widehat{A}_\bullet)$ correspond to the quasi-atomic elements of $D_{S_n}$ that are roots of almost-atomic elements.

Despite the lack of a partial order on the cohomology groups, we obtain examples of interesting factorization properties related to properties like torsion in the cohomology groups.  Example \ref{ex-torsion-homo} presents an integral domain whose quasi-atomic quotient sequence admits a cohomology group with torsion elements in $H^n(G/\widehat{A}_\bullet)$. These elements have a specific factorization interpretation. We use an iterative process of making polynomial extensions as described in \cite{Roi-Poly}, \cite{coy-ap-dom}, and \cite{coy-emb} to construct our example and interpret the result.

\begin{ex}\label{ex-torsion-homo}
Let $K = \bbr$, let $x,y$ be indeterminate over $K$, and consider the set $Y = \left\{y^{\alpha}; \alpha \in \bbq^+\right\}$. In $K[x,Y]$, let $\m$ denote the maximal ideal generated by all monomials, and consider $R_0 = (K[x,Y]_\m)/(x^2 + y^2)$. Since $x^2 + y^2$ is irreducible, this is an integral domain. Furthermore, $x$ is the only irreducible monomial.

Let $A_0$ be the set of all irreducible elements of $R_0$ not associated to $x$. For each $a_0\in A_0$, select a distinct indeterminate over $R_0$, say $w(a_0)$, define the set $\mathcal{W}_0 = \left\{w(a_0), \frac{a_0}{w(a_0)} \mid a_0 \in A_0\right\}$, and define $R_1 = R_0[\mathcal{W}_0]$.
From \cite[Lemma 2.5]{coy-ap-dom}, we have that $U(R_1) = U(R_0)$ so these new divisibility relationship are nontrivial. From \cite[Lemma 2.6]{coy-ap-dom}, $x$ remains irreducible in $R_1$. From \cite[Lemma 2.7]{coy-ap-dom}, any irreducible in $R_0$ not associated to $x$ is reducible in $R_1$.

We proceed inductively. For each $i \geq 1$, let $A_i$ be the set of all irreducible elements of $R_i$ not associate to $x$. For each $a_i \in A_i$, select a distinct indeterminate over $R_i$, say $w(a_i)$. Define the set $\mathcal{W}_i = \left\{w(a_i), \frac{a_i}{w(a_i)} \mid a_i \in A_i\right\}$ and the ring $R_{i+1} = R_i[\mathcal{W}_i]$. The sequence of canonical inclusions
\[R_0 \hookrightarrow R_1 \hookrightarrow \cdots\]
is an ascending chain of integral domains; the direct limit is then simply the union, and is an integral domain: define $R^{\prime} := \underset{\longrightarrow}{\lim} R_i$. In $R^{\prime}$, associates of $x$ are the only irreducible elements, and yet we have the relation $x^2 + y^2 = 0$. Thus, $y^2$ is atomic. 

We claim $y$ is not almost-atomic, which we prove by contradiction. Assume that $y$ is almost-atomic. Then $y$ may be regarded as a ratio of atoms in the field of fractions. Thus, $y = u x^{n}$ for some unit $u$ and $n \in \bbz$. Thus, $y^2 = u^2 x^{2n}$, and we already have that $y^2 = vx^2$ for a unit $v$. Thus, we have that $vx^2 = u^2 x^{2n}$. Thus, we have $x^2(u^2 x^{2n-2} - v) = 0$. Since $R^{\prime}$ is an integral domain, we have that $x^{2n-2}$ is a unit. Of course, $x$ is irreducible in $R^{\prime}$ so we conclude $n=1$ and hence $y = ux$. This contradicts our construction, in which $x$ is irreducible and $y$ is reducible.

Thus, $yU(R^{\prime}) \in G(R^{\prime}) \setminus A(G(R^{\prime}))$, and so is nontrivial in the first cohomology group, and yet its square, $y^2$, is atomic. We conclude $yU(R^{\prime})$ is $2$-torsion in $H^1(G(R^{\prime}), \widehat{A}_\bullet)$.
$\triangle$\end{ex}

Short exact sequences of complexes lead naturally to the long exact sequences in cohomology:

\begin{thm}
Let $1 \longrightarrow X_{\bullet} \longrightarrow Y_{\bullet} \longrightarrow Z_{\bullet} \longrightarrow 1$ denote any of the short exact sequences of cochain complexes from Lemma \ref{three-little-ses}. Then there exists a long exact sequence of group homomorphisms in cohomology:
\begin{align}
H^0(G, X_{\bullet} )\longrightarrow H^0(G, Y_{\bullet})\longrightarrow H^0(G,Z_{\bullet})\longrightarrow H^1(G,X_{\bullet})\longrightarrow\dots	\label{E:eqn22}
\end{align}\label{thm-long-exact}
\end{thm}

This immediately leads to Corollary \ref{cor-hom}, the proof of which requires only a citation of Theorem \ref{thm-long-exact} and Equations \ref{eqn1} - \ref{eqn6}.

\begin{cor}\label{cor-hom}
The following sequences of group homomorphisms are exact.
\begin{align*}
& 1 \to  Q_0 \to Q_0 \to \frac{\widehat{A}_0}{Q_0} \to  Q_1 \to \frac{Q_1}{A_1} \to \frac{\widehat{A}_1}{Q_1} \to  Q_2 \to \frac{Q_2}{A_2} \to \frac{\widehat{A}_2}{Q_2} \to \dots \\
& 1 \to 	 Q_0 \to Q_0 \to \frac{\widehat{Q}_0}{Q_0} \to Q_1 \to 1 \to \frac{\widehat{Q}_1}{Q_1} \to Q_2 \to 1 \to \dots \label{eq-les3}
\end{align*}
In particular, we recover:
\begin{enumerate}[(i)]
\item for any $n \geq 1$, $1 \to A_n \to Q_n \to \frac{Q_n}{A_n} \to 1$ is exact and
\item for any $n \geq 1$, $\frac{\widehat{Q}_{n}}{Q_{n}} \cong Q_{n+1}$.
\end{enumerate}
\end{cor}

\begin{ex}\label{ex-homo}
Let $K$ be a field and consider the 4-dimensional valuation domain
\[
R = \left(K\left[x_1, x_2, x_3, x_4, \frac{x_2}{x_1^j}, \frac{x_3}{x_2^j}, \frac{x_4}{x_3^j}\right]\right)_{\m}
\]
where $j$ ranges over all integers $\geq 1$ and $\m$ is the maximal ideal generated by all indeterminate elements over $K$. This domain has group of divisibility o-isomorphic to $\bbz^4$ ordered lexicographically; in particular, we have a single irreducible, $x_1$. After the first localization, we have only one irreducible, $\frac{x_2}{1}$.  Localizing twice yields a single irreducible, $\frac{x_3}{1}$, localizing a third time yields a single irreducible, $\frac{x_4}{1}$, and localizing a fourth time yields the quotient field.

Observe, however, that our quasi-atomic subgroup coincides with our atomic subgroup, which is o-isomorphic to $\bbz$. Hence, our quasi-atomic quotient sequence is o-isomorphic to the sequence (with each product ordered lexicographically):
\[ \bbz^4 \to \bbz^3 \to \bbz^2 \to  \bbz \to 0 \to 0 \to \dots\]
and thus, by Definition \ref{defn-atomic}, $R$ is $4$-atomic. We resolve this sequence in detail as in Figure \ref{fig:detailedCommutativeDiagram}.
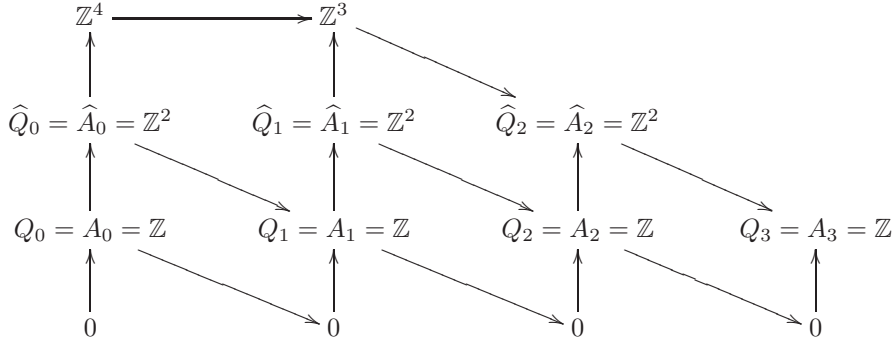
\begin{figure}
\xymatrix{
\bbz^4 \ar[r] & \bbz^3 \ar[dr] & & \\
\widehat{Q}_0 = \widehat{A}_{0} = \bbz^2 \ar[u] \ar[dr] & \widehat{Q}_1 =\widehat{A}_{1} = \bbz^2 \ar[u] \ar[dr] & \widehat{Q}_2 = \widehat{A}_{2} = \bbz^2 \ar[dr] &\\
Q_0 = A_0 = \bbz \ar[u] \ar[dr] & Q_1 = A_1 = \bbz \ar[u] \ar[dr] & Q_2 = A_2 = \bbz \ar[u] \ar[dr] & Q_3 = A_3 = \bbz \\
0 \ar[u] & 0 \ar[u] & 0 \ar[u] &0 \ar[u]
}
\caption{A detailed expansion of the quasi-atomic quotient sequence for the ($4$-atomic) po-group presented in Example \ref{ex-homo}}\label{fig:detailedCommutativeDiagram}
\end{figure}

Hence, Corollary \ref{cor-hom} only yields one distinct long exact sequence:
\[\xymatrix{
0 \to  Q_0=\bbz \ar[r] & Q_0=\bbz \ar[r] &  \frac{\widehat{A}_0}{Q_0} = \frac{\bbz^2}{\bbz} \ar[r] &  Q_1 = \bbz \ar[r] &  \frac{Q_1}{A_1} = 0 \ar[dllll] \\
\frac{\widehat{A}_1}{Q_1} = \frac{\bbz^2}{\bbz} \ar[r] &  Q_2=\bbz \ar[r] &  \frac{Q_2}{A_2}=0 \ar[r] & \frac{\widehat{A}_2}{Q_2} = \frac{\bbz^2}{\bbz} \ar[r] & Q_3 = \bbz \ar[r] & 0 
}\]
More concisely, we have
\[0 \to \bbz \overset{=}{\longrightarrow} \bbz \overset{0}{\longrightarrow} \frac{\bbz^2}{\bbz} \overset{\cong}{\longrightarrow} \bbz \to 0 \to \frac{\bbz^2}{\bbz} \overset{\cong}{\longrightarrow} \bbz \to 0 \to \frac{\bbz^2}{\bbz} \overset{\cong}{\longrightarrow} \bbz \to 0 \to 0 \to \dots\]
Of course, this long exact sequence breaks into four fundamental short exact sequences o-isomorphic to the form $0 \to \bbz \to \bbz \to 0$.
$\triangle$\end{ex}

%%%%%%%%%%%%%%%%%%%%%%%%%%%%%%%%%%%%%%%%%%%%%%%%%%%%%%%%%%%%
\section{Structure Within Partially Ordered Abelian Groups}\label{structure}
%%%%%%%%%%%%%%%%%%%%%%%%%%%%%%%%%%%%%%%%%%%%%%%%%%%%%%%%%%%%

We elaborate upon the structure of groups of divisibility. Perhaps we wish to write every po-group as the group of divisibility of an integral domain. This is not always possible, as every group of divisibility is directed, and it is easy to construct po-groups that are not directed:

\begin{ex}\label{generatedbypositive}
Let $G=\mathbb{Q}$ with partial ordering $a\leq b \Leftrightarrow b-a\in \bbn$. The positive elements of $G$ are precisely $\bbn$. Hence $A(G)=Q(G)=\mathbb{Z}$. Note that no direct sum decomposition is possible, since $\mathbb{Z}$ is not a direct summand of $\mathbb{Q}$. However, this po-group under the given order is not generated by its positive elements and hence is not a group of divisibility for any domain.
$\triangle$\end{ex}

When we are provided the luxury that our po-group is a group of divisibility, we may wish to decompose $G(D)$ into a product (sum, coproduct, etc.) of other po-groups or other groups of divisibility. In particular, we may be inclined to guess that every group of divisibility decomposes into an internal direct sum of the quasi-atomic subgroup with some direct sum complement, i.e.\ a sum of an atomic part with an antimatter part. However, this cannot be expected in general: we provide an example of an integral domain whose group of divisibility may not be written as a sum of an atomic part and an antimatter part. We, again, use the results of \cite{Roi-Poly}, \cite{coy-ap-dom}, and \cite{coy-emb} to construct this example, by using a candidate integral domain, and then constructing a chain of integral domains whose direct limit satisfies our desired properties.

\begin{ex}\label{ex-no-split}
Recall the ring in Example \ref{ex-torsion-homo}  $R_0 = (K[x,Y]_\m)/(x^2+y^2)$, where $Y = \left\{y^{\alpha} \mid \alpha \in \bbq^+\right\}$ and where $x$ is irreducible. We have proven $y$ is not almost-atomic (that proof can be extended to demonstrate that, in fact, $y^{\alpha}$ is not almost-atomic for any $0 < \alpha < 2$). However, $x$ is an atom. Hence, since $x^2 + y^2 = 0$ in $R_0$, each element $y^{\alpha}$ for $0 < \alpha < 2$ is quasi-atomic but not almost-atomic, and, furthermore, there exists a power $k \in \bbn$ such that $(y^{\alpha})^k$ is atomic. In particular, we cannot write the set of non-atomic elements of $G(R_0)$ as a summand of $G(R_0)$.
\end{ex}

The hypothesis that every group of divisibility splits into a direct sum of po-groups, then, is demonstrated to be false by Example \ref{ex-no-split}: the nontrivial intersection between atomic and non-atomic parts of the domain kills all our hopes of writing the domain as a direct sum. Recall a short exact sequence of groups splits when the associated surjection has a right inverse. We expand our inquiry to ask when groups of divisibility (or perhaps more generally partially ordered abelian groups) split.

\begin{lem}
\label{structure_theorem_weak}
Let $H$ be a convex subgroup of a po-group $G$. Consider the natural short exact sequence of po-group o-homomorphisms:
\[\xymatrix{
0 \ar[r] & H \ar[r]^-{\iota} & G \ar[r]^-{\pi} & G/H \ar[r] & 0
}\]
If there exists an o-homomorphism of po-groups acting as a right inverse for $\pi$, say $\pi \circ \pi^{-1} = \text{id}_{G/H}$, then $\frac{G}{H} \oplus H$ is a po-group under the partial order given by $(g_1 + H, h_1) \leq (g_2 + H, h_2)$ if and only if $h_1 + \pi^{-1}(g_1 + H) \leq h_2 + \pi^{-1}(g_2+H)$ in $G$. The product order on $\frac{G}{H} \oplus H$ is finer than this induced splitting order. Furthermore, under this partial order, there exists a po-group o-isomorphism $\phi: G \to \frac{G}{H} \oplus H$ commutative in the diagram
\[\xymatrix{0 \ar[r] & H \ar[r]^-{\iota} \ar[d]^-{=} & G \ar[d]^-{\phi} \ar[r]^-{\pi}&  G/H \ar[r] \ar[d]^-{=} & 0\\
0 \ar[r] & H \ar[u] \ar[r]^-{\iota^{\prime}} & \frac{G}{H} \oplus H \ar[r]^-{\pi^{\prime}} & G/H \ar[u] \ar[r] & 0
}\]
\end{lem}
\begin{proof}
Since $\pi$ has a right inverse as a group homomorphism, we have that $G$ splits as an abelian group. In particular, we have a pair of group isomorphisms, $\phi: G \to \frac{G}{H} \oplus H$ given by $g \mapsto (\pi(g), g - \pi^{-1} \circ \pi(g))$ and $\psi: \frac{G}{H} \oplus H \to G$ given by $(g+H,h) \mapsto h + \pi^{-1}(g+H)$.

We verify that the relation indicated in the lemma is a partial order on $\frac{G}{H}\oplus H$. We have reflexivity since, if $(g+H, h) \in \frac{G}{H}\oplus H$, then $h + \pi^{-1}(g+H) \leq h + \pi^{-1}(g+H)$ in $G$. Transitivity is assured due to transitivity in $G$: if $(g_1 + H, h_1) \leq (g_2 + H, h_2)$ and $(g_2 + H, h_2) \leq (g_3 + H, h_3)$, then we have that $h_1 + \pi^{-1}(g_1 + H) \leq h_2 + \pi^{-1}(g_2 + H)$ and $h_2 + \pi^{-1}(g_2 + H) \leq h_3 + \pi^{-1}(g_3 + H)$. Thus, $h_1 + \pi^{-1}(g_1 + H) \leq h_3 + \pi^{-1}(g_3 + H)$.

We also have antisymmetry. To this end, let $(g_1 + H, h_2) \leq (g_2 + H, h_2)$ in $\frac{G}{H} \oplus H$ and vice versa. Then $h_1 + \pi^{-1}(g_1 + H) \leq_G h_2 + \pi^{-1}(g_2 + H)$ and vice versa. By antisymmetry in $\leq_G$, we have that $h_1 + \pi^{-1}(g_1 + H) = h_2 + \pi^{-1}(g_2 + H)$. In particular, since $\pi^{-1}$ is an o-homomorphism, $h_2 - h_1 = \pi^{-1}(g_1 - g_2 + H) \in H$. Hence, $\pi^{-1}(g_1 - g_2 + H) \in H$. But then $\pi \circ \pi^{-1}\left(g_1 - g_2 + H\right) = H$. Since $\pi^{-1}:\frac{G}{H} \to G$ is the right inverse of $\pi: G \to G/H$, we have that $\pi \circ \pi^{-1} = \text{id}_{G/H}$, so we have that $g_1 - g_2 + H = H$ and therefore $g_1 + H = g_2 + H$. Now $\pi^{-1}(g_1 + H) = \pi^{-1}(g_2 + H)$ so we obtain the following.
\begin{align*}
    h_1 + \pi^{-1}(g_1 + H) =& h_2 + \pi^{-1}(g_2 + H)\\
    h_1 + \pi^{-1}(g_1 + H) =& h_2 + \pi^{-1}(g_1 + H)\\
    h_1 =& h_2
\end{align*}
Hence, $(h_1, g_1 + H) = (h_2, g_2 + H)$, demonstrating antisymmetry.

We verify that $\frac{G}{H} \oplus H$ is a po-group under the order induced by the orders on $G/H$ and $H$ (i.e.\ the group operation is compatible). Let $(g_3 + H, h_3) \in \frac{G}{H} \oplus H$. Then, if $(g_1 + H, h_1) \leq (g_2 + H, h_2)$ then
\[h_1 + \pi^{-1}(g_1 + H) \leq h_2 + \pi^{-1}(g_2 + H)\] 
and so
\[(h_1 + \pi^{-1}(g_1 + H))  + (h_3 + \pi^{-1}( g_3 + H)) \leq h_2 + \pi^{-1}(g_2+ H) + (h_3 + \pi^{-1}( g_3 + H))\]
Rearrange this to write
\[(h_1 + h_3) + \pi^{-1}(g_1 + g_3 + H) \leq (h_2 + h_3) + \pi^{-1}(g_2 + g_3 + H)\]
which, in turn, implies that $(g_1 + g_3 + H, h_1 + h_3) \leq (g_2 + g_3 + H, h_2 + h_3)$ in $\frac{G}{H}\oplus H$, i.e.\ $(g_1 + H, h_1) + (g_3 + H, h_3) \leq (g_2 + H, h_2) + (g_3 + H, h_3)$. Hence, $\frac{G}{H} \oplus H$ is a po-group under the induced order.

The product order is finer than this partial order: if $(g_1+H, h_1) \leq (g_2 + H, h_2)$ in the product order, then $g_1 + H \leq g_2 + H$ and $h_1 \leq h_2$. Hence we have that $h_1 + \pi^{-1}(g_1+H) \leq h_2 + \pi^{-1}(g_2+H)$. 

Lastly, that $\phi$ is an o-isomorphism is also easy to verify: we already have that $\phi$ is an isomorphism so we must simply verify that it is order-preserving and an o-epimorphism. Note that for any $g_1, g_2 \in G$, we have that 
\[(\pi(g_1),g_1 - \pi^{-1}\circ\pi(g_1)) \leq (\pi(g_2),g_2 - \pi^{-1}\circ\pi(g_2))\] 
which is true if and only if 
\[g_1 - \pi^{-1}\circ \pi(g_1) + \pi^{-1}\circ \pi(g_1) \leq g_2 - \pi^{-1}\circ \pi(g_2) + \pi^{-1} \circ \pi(g_2)\] 
This reduces to $g_1 \leq g_2$. Hence, the map $g \mapsto (\pi(g), g - \pi^{-1}\circ \pi(g))$ is order-preserving. It remains to be shown that this map is an o-epimorphism.

Of course, for any positive $x \in (\frac{G}{H} \oplus H)^+$, we have that $x = (g+H, h)$ for some $g \in G$ and $h \in H$ such that $0 \leq h + \pi^{-1}(g+H)$. Moreover, if $x = \phi(g)$ for some $g$, then $x = (g+H, g - \pi^{-1}\circ \pi(g))$. Since $\phi$ is a group isomorphism, the map is surjective so $x=(g+H,h) = (\pi(g_0), g_0 - \pi^{-1}\circ \pi(g_0))$ for some $g_0 \in G$. By the partial order induced on $\frac{G}{H} \oplus H$ by $G$, we have that $0 \leq (\pi(g_0), g_0 - \pi^{-1}\circ \pi(g_0)) $ if and only if $0 \leq g_0 - \pi^{-1}\circ \pi(g_0) + \pi^{-1}\circ \pi(g_0) = g_0 $. Hence, $x$ is the image of the positive element $g_0$, so $\phi(G^+) = (\frac{G}{H} \oplus H)^+$.

This proves that we have an o-isomorphism $G \simeq \frac{G}{H} \oplus H$, where this direct sum is  under the induced partial order $(g_1 + H, h_1) \leq (g_2 + H, h_2)$ if and only if $h_1 + \pi^{-1}(g_1+H) \leq h_2 + \pi^{-1}(g_2+H)$.
\end{proof}

\begin{defn}
We call the partial order induced on $\frac{G}{H} \oplus H$ from Lemma \ref{structure_theorem_weak} the \textit{induced splitting partial order}. 
\end{defn}

Lemma \ref{structure_theorem_weak} can be immediately applied to the quasi-atomic quotient sequence.

\begin{cor}\label{cor:fin-sums}
Let $G_0 \overset{\pi_0}{\rightarrow} G_1 \overset{\pi_1}{\rightarrow} G_2 \overset{\pi_2}{\rightarrow} \cdots$ be a quasi-atomic quotient sequence terminally stable at degree $N$ for some po-group $G$. If for each $0 \leq n < N-1$ there exists an o-homomorphism $\phi_n:G_{n+1} \to G_n$ that is a right-inverse for $\pi_n$, then there exists an o-isomorphism $\psi: \oplus_{n=N-1}^{0} Q(G_n) \to G$.
\end{cor}
\begin{proof}
If $N=0$ then $G_0$ is trivial and we are done. If $N=1$, then the quasi-atomic quotient sequence is $G_0 \to \left\{e\right\}$. In particular $Q(G_0) = G_0$. Let $N=2$. Then we have the quasi-atomic quotient sequence $G_0 \overset{\pi_0}{\onto} G_0/Q(G_0) \overset{\pi_1}{\onto} 0$ and a right inverse for $\pi_0$, which is an o-homomorphism $\pi^{-1}_0: G_0/Q(G_0) \to G_0$. Then the short exact sequence
\[0 \to Q(G_0) \to G_0 \to G_0/Q(G_0) \to 0\]
is split exact so $G_0 \simeq \frac{G_0}{Q(G_0)} \oplus Q(G_0)$. Note, however, that $G_2 = G_1/Q(G_1)$, which (in this instance) is trivial. Hence, $Q(G_1) = G_1 = \frac{G_0}{Q(G_0)}$ so we have that $G_0 \simeq Q(G_1) \oplus Q(G_0)$.  Note the ordering on the direct sum is reversed. Inductively we obtain the result for any finite bounded quasi-atomic quotient sequence.
\end{proof}

Theorem \ref{structure_theorem_weak} verifies that the group isomorphism we obtain from the usual splitting theorem is, indeed, an o-isomorphism. Assuming $G$ splits as described in Theorem \ref{structure_theorem_weak}, it is natural to ask when the induced partial order on $\frac{G}{H} \oplus H$ coincides with the product or the lexicographic order. Mo\v{c}ko\v{r} established the answer to the lexicographic component of this question in \cite{mockor1983groups}, and here we extend that result to establish the product order component of this question:

\begin{thm}
\label{structure_theorem_strong}
Let $H$ be a convex subgroup of a po-group $G$, let $\pi:G \to G/H$ the canonical o-epimorphism, and let $\pi^{-1}$ be an o-homomorphism right inverse of $\pi$ so $\pi \circ \pi^{-1} = \text{id}_{G/H}$. Then $G \simeq \frac{G}{H} \oplus H$ under the induced splitting partial order. Moreover, (i) if $G^+ = \left\{g \in G \mid \pi(g) > H\text{ or }g \in H^+\right\}$ then the induced splitting partial order coincides with the lexicographic partial order and (ii) if $G^+ = \left\{g \in G \mid g - \pi^{-1}\circ \pi(g) \in H^+\right\}$, then the induced splitting partial order on $\frac{G}{H}\oplus H$ coincides with the product partial order.
\end{thm}
\begin{proof}
Let $\phi: G \to \frac{G}{H} \oplus H$ be defined, as usual, by $g \mapsto (\pi(g), g - \pi^{-1}\circ \pi(g))$. It has been established that $\phi$ is a group isomorphism, and $\phi$ is an o-isomorphism when $\frac{G}{H} \oplus H$ has the induced splitting partial order. It is sufficient to show that, if $G$ satisfies the assumed property, then the induced splitting partial order is equivalent to the lexicographic order.

First we show that if $G^+ = \left\{g \in G \mid \pi(g) > H\text{ or }g \in H^+\right\}$ then the lexicographic order is finer and coarser than the induced splitting order. To show the lexicographic order is finer than the splitting order, let $(g_1 + H, h_1) \leq (g_2 + H, h_2)$ in $\frac{G}{H} \oplus H$ under the lexicographic order. Then $g_1 + H < g_2 + H$ or $g_1 + H = g_2 + H$ and $h_1 \leq h_2$. In particular, $g_1 + H < g_2 + H$ or $g_1 + H = g_2 + H$ and $h_1 \leq h_2$. If $g_1 + H < g_2 + H$, we have that $H < g_2 - g_1 + H$. By our assumption on $G$, if $\alpha \in g_2 - g_1 + H$, then $\alpha + H > H$ so $\alpha \in G^+$. Also, $h_2 - h_1 + \pi^{-1}(g_2 - g_1 + H) \in g_2 - g_1 + H$ so $(g_1 + H, h_1) \leq (g_2 + H, h_2)$ under the splitting order.  On the other hand, if $g_1 + H = g_2 + H$ and $h_1 \leq h_2$, then $\pi^{-1}(g_2 - g_1 + H) = \pi^{-1}(H) = 0$, and so $h_2 - h_1 + \pi^{-1}(g_2 - g_1 + H) = h_2 - h_1 \geq 0$.  Hence, $(g_1 + H, h_1) \leq (g_2 + H, h_2)$ under the splitting order. 

To show that the induced splitting order is finer than the lexicographic, say $(g_1 + H, h_1) \leq (g_2 + H, h_2)$ in $\frac{G}{H} \oplus H$ under the splitting order. Then we have $h_1 + \pi^{-1}(g_1 + H) \leq h_2 + \pi^{-1}(g_2 +H)$ in $G$ and since $\pi^{-1}$ is an o-homomorphism, $\pi^{-1}(g_2 + H) - \pi^{-1}(g_1 + H) = \pi^{1}(g_2 - g_1 + H)$. Thus, $\pi^{-1}(g_2 - g_1 + H) + (h_2 - h_1) \geq 0$. Denote $g:= \pi^{-1}(g_2 - g-1 + H) + (h_2 - h_1)$. Note $g \geq 0$. By our assumption on $G$, we have that $g+H > H$ or $g \in H^+$. If $g + H > H$, then applying $\pi$ to both sides of $\pi^{1}(g_2 - g_1 + H) + (h_2 - h_1)$ yields that $g_2 - g_1 + H > H$, i.e.\ $g_2 + H > g_1 + H$. Hence, $(g_1 + H, h_1) \leq (g_2 + H, h_2)$ under the lexicographic order. On the other hand, if $g \in H^+$, then $\pi^{-1}(g_2 - g_1 + H) + (h_2 - h_1) \geq 0$ in $H$. Applying $\pi$ to both sides reveals that $g_2 - g_1 + H = H$, so$g_1 + H = g_2 + H$. Thus, we have that $g = h_2 - h_1 + \pi^{-1}(g_2 - g_1 + H) = h_2 - h_1 \geq 0$, so $h_1 \leq h_2$, and so $(g_1 + H, h_1) \leq (g_2 + H, h_2)$ in the lexicographic order.

For the product order version of the result, let $\phi: G \to \frac{G}{H} \oplus H$ be defined, as usual, by $g \mapsto (\pi(g), g - \pi^{-1}\circ \pi(g))$. It has been established that $\phi$ is a group isomorphism, and $\phi$ is an o-isomorphism when $\frac{G}{H} \oplus H$ has the induced splitting partial order. It is sufficient to show that if $G^+ = \left\{g \in G \mid g - \pi^{-1}\circ \pi(g) \in H^+\right\}$ then the induced splitting partial order is equivalent to the product order.

First we show that if $G^+ = \left\{g \in G \mid g - \pi^{-1}\circ \pi(g) \in H^+\right\}$ then the product order is finer than the induced splitting order. Say that $(g_1 + H, h_1) \leq (g_2 + H, h_2)$ in $\frac{G}{H} \oplus H$ under the product order so that $g_1 + H \leq g_2 + H$ and $h_1 \leq h_2$. We have $g_2 - g_1 + H \geq  H$ so $\pi^{-1}(g_2 - g_1 + H) \geq 0$. Of course, since $h_2 - h_1 \geq 0$ we have  $(h_2 - h_1) + \pi^{-1}(g_2 - g_1 + H) \geq 0$. Thus, $(g_1 + H, h_1) \leq (g_2 + H, h_2)$ in the induced splitting order on $G$.  

All that remains is to show the induced splitting order is finer than the product order. Let $(g_1 + H, h_1) \leq (g_2 + H, h_2)$ in the induced splitting order on $\frac{G}{H} \oplus H$ so that $h_1 + \pi^{-1}(g_1 + H) \leq h_2 + \pi^{-1}(g_2 + H)$. In particular, we have the non-negativity of $x:=h_2 - h_1 + \pi^{-1}(g_2 - g_1 + H) \geq 0$. Since $\pi$ is order preserving, we have $\pi(x) > H$, and by our assumption on $G$, we have $x - \pi^{-1}\circ \pi(x) \geq 0$. Of course, $\pi(x) = \pi(h_2 - h_1 + \pi^{-1}(g_2 - g_1 + H)) = g_2 - g_1 + H$, and $x - \pi^{-1}\circ \pi(x) = h_2 - h_1$, and so $(g_1 + H, h_1) \leq (g_2 + H, h_2)$ under the product order.
\end{proof}

Theorem \ref{structure_theorem_strong} provides the required conditions to determine when the induced splitting order coincides with the product order or the lexicographic order. Iteratively applying Corollary \ref{cor:fin-sums} yields a decomposition Theorem for quasi-atomic quotient sequences:

\begin{thm}
Let $G$ be a group of divisibility with quasi-atomic quotient sequence such that each non-zero differential, $\pi_n : G_n \to G_{n+1} = G_n/Q(G_n)$, has a right inverse o-homomorphism. Then each $G_n \simeq G_{n+1} \oplus Q(G_n)$ under the induced splitting order and we have the following:
\begin{enumerate}[(i)]
\item if $G$ is $n$-atomic, then $G\simeq\oplus_{i=n-1}^{0}Q(G_{i})$ and
\item if $G$ is $n$-antimatter, then $G\simeq G_n\oplus\left(\oplus_{i=n-1}^{0}Q(G_{i})\right)$.
\end{enumerate}
Furthermore, if each $G_n^+ = \left\{g \in G_n \mid \pi_n(g) > Q(G_{n})\text{ or }g \in Q(G_n)^+\right\}$ then the induced splitting partial order coincides with the lexicographic partial order. If each $G_n^+ = \left\{g \in G_n \mid g - \pi_n^{-1}\circ \pi_n(g) \in Q(G_n)^+\right\}$, then the induced splitting partial order coincides with the product partial order.
\label{cor-useful}
\end{thm}

We omit the proof of this theorem (it follows directly from \ref{cor:fin-sums} and \ref{structure_theorem_strong}). We intuitively think Theorem \ref{cor-useful} as stating: if the $n$-atomic and splits at every degree, then every element of $G$ may be uniquely written as a sum of terms, one from each quasi-atomic part of the sequence. This provides a weak version of a universal factorization property. Similarly, $n$-antimatter group have a sort of universal property as well: if the quasi-atomic sequence is $n$ anti-matter every element of $G$ may be uniquely written as a sum of terms, one from each quasi-atomic part of the sequence and one term from the stable target of the sequence, the antimatter po-group $G_n$. Example \ref{ex:ring-examples}  presents integral domains whose groups of divisibility satisfy the stability properties from Definition \ref{defn-atomic}.

\begin{ex}\label{ex:ring-examples}
\begin{enumerate}[(i)]
\item In this example, we present a ring-theoretic realization of Example \ref{ex-n-atomic}(i), which was a po-group with an $n$-atomic quasi-atomic quotient sequence. Let $K$ be a field and consider the 2-dimensional valuation domain
\[R:= \left(K\left[x_1,x_2, \frac{x_2}{x_1^j}\right]\right)_{\m}\]
\noindent where $j$ ranges over all integers $j \geq 1$ and we have localized at the maximal ideal $\m$ generated by the set of all indeterminate elements. The group of divisibility is $G(R) \simeq \bbz^2$ in the lexicographic order. In particular, $R$ has precisely one irreducible, $x_1$; by localizing at $x_1$, we obtain a new ring in which $\frac{x_2}{1}$ is the only irreducible, and localizing again yields the quotient field. Hence, the ring is $2$-atomic.   After the first stage of localizing, the new group of divisibility is $1$-atomic and the atomic subgroup is now o-isomorphic to $\bbz$ since it, too, has only one irreducible.

The quasi-atomic quotient sequence is o-isomorphic to $\bbz^2 \to \bbz \to 0$. $\triangle$

\item In this example, we present a ring-theoretic realization of Example \ref{ex-n-atomic}(ii), a po-group with an $n$-antimatter quasi-atomic quotient sequence. Let $K$ be a field with indeterminates $x, y$, let $Y = \left\{y^{\alpha}, \frac{y^{\alpha}}{x^j} \mid \alpha \in \bbq^+right, j \in \bbn\right\}$, and consider $R = (K[x,Y])_{\m}$ where $\m$ is the maximal ideal generated by all monomials. Every element of $R$ is associated to some monomial, and $x$ is uniquely irreducible (up to associates) among all monomials. The group of divisibility for this ring is $\bbz \oplus \bbq$ ordered lexicographically. We have the quasi-atomic quotient sequence $\bbz \oplus \bbq \onto \bbq \onto \bbq \onto \cdots$.  This example naturally extends to an $2$-antimatter group of divisibility by considering the ring $\left(K\left[y^{\alpha}, x, z, \frac{y^{\alpha}}{x^j}, \frac{x}{z^j}\right]\right)_{\m}$; clearly, this example will extend inductively to any $n$-antimatter group as necessary $\triangle$

\item We present an example of a domain whose group of divisibility has a quasi-atomic quotient sequence that is not stable but is o-isomorphic in each degree. Let $K$ be a field and consider the infinite dimensional valuation domain
\[R:= \left(K\left[x_1,x_2,x_3,\dots,\frac{x_2}{x_1^j},\frac{x_3}{x_2^j},\dots,\frac{x_m}{x_{m-1}^{j}},\dots\right]\right)_{\m}\]
where $j$ ranges over all integers $j \geq 1$ and we have localized at the maximal ideal $\m$ generated by the set of all indeterminate elements. The group of divisibility is $G(R) \simeq \oplus_{i\in \bbn} \bbz$. In particular, $R$ has precisely one irreducible, $x_1$, which is prime; by localizing at $x_1$, we obtain a new ring in which $\frac{x_2}{1}$ is the only irreducible, which is prime, and so on. The domain is mixed antimatter and atomic. We have that $Q(G) = A(G) \simeq \bbz \oplus (\oplus_{i > 1} 0) \simeq \bbz$. Of course, we have that $\frac{\oplus_{i \in \bbn} \bbz}{\bbz \oplus 0 \oplus 0 \oplus \cdots} \simeq \oplus_{i \in \bbn} \bbz$, so each o-epimorphism in the quasi-atomic quotient sequence is an o-isomorphism. We have the quasi-atomic quotient sequence
\[ \oplus_{i \geq 1} \bbz \to \left[(\oplus_{i=1}^{1} 0) \oplus (\oplus_{i \geq 2} \bbz)\right] \to \left[(\oplus_{i=1}^{2} 0) \oplus (\oplus_{i \geq 3} \bbz)\right] \to \cdots\]
which is o-isomorphic to the sequence
\[ \oplus_{i \in \bbn} \bbz \to \oplus_{i \in \bbn} \bbz \to \oplus_{i \in \bbn} \bbz \to \dots \]
Note even though all degrees are o-isomorphic, the o-epimorphisms in the quasi-atomic quotient sequence are not o-isomorphisms. $\triangle$
\end{enumerate}
\end{ex}

\begin{ex}\label{ex:branching-primes}
We construct a domain with a non-stabilizing, non-terminating quasi-atomic quotient sequence that has po-groups in each degree that are not o-isomorphic. The idea is to construct an integral domain whose graph of prime ideals is a tree with one root. After the first stage of localization, we wish to obtain an integral domain for which this tree now has two roots, corresponding to two co-maximal ideals. We continue this process such that after the $n^{th}$ stage of localization, we obtain a domain whose graph of prime ideals is a tree with $n+1$ roots as co-maximal ideals.  To this end, let $K$ be a field with indeterminate $\left\{x_{i,j} \mid 1 \leq i, 1 \leq j \leq i\right\}$. Let
\[T:=K\left[x_{1,1}, x_{2,1},x_{2,2}, \frac{x_{2,i}}{x_{1,j}^{\ell}}, x_{3,1}, x_{3,2}, x_{3,3}, \frac{x_{3,i}}{x_{2,j}^{\ell}}, x_{4,1},x_{4,2},x_{4,3},x_{4,4}, \frac{x_{4,1}}{x_{3,j}^{\ell}}, \ldots \right]\]
where the index $i$ in the terms $x_{k,i}$ varies from $1 \leq i \leq k$, the index $j$ in the terms $x_{k,j}$ range from $1 \leq j \leq k$, and where each $\ell$ ranges over $\ell \geq 1$. That is, we have the divisibility conditions: $x_{1,1}^{\ell} \mid x_{2,i}$ for every $1 \leq i \leq 2$ and every $1 \leq \ell$, we have $x_{2,1}^{\ell} \mid x_{3,i}$ for every $1 \leq i \leq 3$ and every $1 \leq \ell$, and so on. These divisibility conditions ensure the graph of prime ideals is the branching tree of co-maximal ideals as desired. Define the prime ideals
\begin{align*}
\p_{1,1} &= \left(x_{1,1}\right) & \p_{3,1} &= \left(x_{3,1}, \frac{x_{3,1}}{x_{2,1}^{\ell}}, \frac{x_{3,1}}{x_{2,2}^{\ell}}\right)\\
\p_{2,1} &= \left(x_{2,1}, \frac{x_{2,1}}{x_{1,1}^{\ell}}\right) & \p_{3,2} &= \left(x_{3,2}, \frac{x_{3,2}}{x_{2,1}^{\ell}}, \frac{x_{3,2}}{x_{2,2}^{\ell}}\right)\\
\p_{2,2} &= \left(x_{2,2}, \frac{x_{2,2}}{x_{1,1}^{\ell}}\right) & \p_{3,3} &= \left(x_{3,3}, \frac{x_{3,3}}{x_{2,1}^{\ell}}, \frac{x_{3,3}}{x_{2,2}^{\ell}}\right)
\end{align*}
and so on for $\p_{i,j}$. Further define the set $S:=\left(\displaystyle \cup_{1 \leq j \leq i}^{\infty} \p_{i,j} \right)^{c}$.  The graph of prime ideals for $T$ this ring is depicted in Figure \ref{fig:prime-ideal-tree}. The prime ideal $\p_{1,1}$ is principal. Now consider the ring $R=T_{S}$. In this ring, $x_{1,1}$ is prime and uniquely irreducible. The graph of prime ideals for $R$ this ring is depicted in Figure \ref{fig:prime-ideal-tree-after}. The co-maximal prime ideals $(\p_{2,1})_S$ and $(\p_{2,2})_S$ are principal. In fact, at every stage of localization, the collection of co-maximal prime ideals are all principal. 

\begin{figure}
\centerline{
\xymatrix{
 & &  \p_{1,1}   & &\\
 &  \p_{2,1} \ar[ur] & & \p_{2,2} \ar[ul]  &\\
 \p_{3,1}  \ar[ur] \ar[urrr]  & & \p_{3,2} \ar[ul] \ar[ur] & & \p_{3,3} \ar[ul] \ar[ulll]\\
\vdots \ar[u] & & \vdots \ar[u]& & \vdots \ar[u] \\
}
}
\caption{The graph of prime ideals of the ring $T$ from Example \ref{ex:branching-primes}. Here arrows denote set inclusion.}
\label{fig:prime-ideal-tree}
\end{figure}
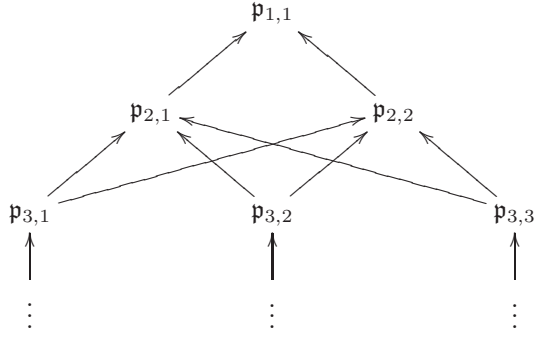

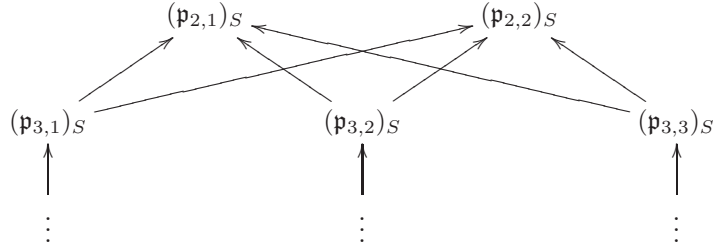
\begin{figure}
\centerline{
\xymatrix{
&  (\p_{2,1})_S  & & (\p_{2,2})_S  &\\
(\p_{3,1})_S  \ar[ur] \ar[urrr]  & & (\p_{3,2})_S \ar[ul] \ar[ur] & & (\p_{3,3})_S \ar[ul] \ar[ulll]\\
\vdots \ar[u] & & \vdots \ar[u]& & \vdots \ar[u] \\
}
}
\caption{The graph of prime ideals of the ring $T_S$ from Example \ref{ex:branching-primes}. Here arrows denote set inclusion.}
\label{fig:prime-ideal-tree-after}
\end{figure}

After the first stage of localization in the process of generating the quasi-atomic quotient sequence, we lose the prime ideal $\p_{1,1}$. Indeed, the generator of $\p_{1,1}$ becomes a unit.  The prime ideals $\p_{2,1}$ become comaximal and principal and the new ring after localization has precisely two irreducible elements, $x_{2,1}$ and $x_{2,2}$. At each stage of localization the (principal) generators of the current top layer become units, revealing the next layer. After the second stage of localization, the ring has three irreducible elements, $x_{3,1}, x_{3,2}$, and $x_{3,3}$, and the graph of prime ideals is a tree with three roots. The process continues iteratively. At the $n^{th}$ stage, we have $n+1$ distinct irreducible elements, corresponding to a quasi-atomic subgroup o-isomorphic to $\bbz^{n+1}$ in the product order. 

For each $n \geq 1$, define $H_n = \oplus_{i=1}^{n} \bbz$ in the product order. Then the group of divisibility for this ring is o-isomorphic to $G = \oplus_{i=1}^{\infty} H_i$ ordered lexicographically.   Then $Q(G) = H_1 \oplus 0 \oplus 0 \oplus \cdots \oplus 0$, so $G_1 = G_0/Q(G_0) \simeq \oplus_{i=2}^{\infty} H_i$. Inductively, we see that
\[\oplus_{i\geq 1} H_i \to \oplus_{i \geq 2} H_i \to \oplus_{i \geq 3} H_ \to \cdots \]
is the quasi-atomic quotient sequence. $\triangle$
\label{notStableNotTerminating}
\end{ex}

Finally, we complete this section by formalizing the notion of the ``antimatter part'' of a ring to which we have been alluding throughout this entire section.

\begin{defn}
Let $G$ be a po-group with $Q(G)$ be the quasi-atomic subgroup of $G$. We define a \textit{quasi-atomic complement} of $G$ to be any subgroup $H \subseteq G$ such that $H\cap Q(G)=0$ that also satisfies the property that, for all $g\in G$, there is an $n\in\bbn$ such that $ng\in H\oplus Q(G)$. We define a \textit{maximal quasi-atomic complement} of $G$ to be any quasi-atomic complement maximal among the set of all quasi-atomic complements of $G$ with respect to inclusion of subgroups.
\end{defn}

\begin{lem}
Let $G$ be a po-group with quasi-atomic subgroup $Q(G)$. Then a maximal quasi-atomic complement of $G$ exists.\label{lem-complement}
\end{lem}
\begin{proof}Let $\mathfrak{J}$ denote the set of all subgroups of $G$ with a trivial intersection with the quasi-atomic subgroup of $G$; that is to say, set $\mathfrak{J} = \left\{J \subseteq G \mid J \cap Q(G) = \langle 0\rangle \right\}$.  Since $\langle 0\rangle  \in \mathfrak{J}$, we have that $\mathfrak{J}$ is nonempty.  Any chain in $\mathfrak{J}$, say $\left\{J_\lambda\right\}_{\lambda \in \Lambda}$, certainly has an upper bound, namely $\cup_\lambda J_\lambda$.  Applying Zorn's Lemma yields a maximal element of $\mathfrak{J}$, which we denote $H^{\prime}$. We claim that $H^{\prime}$ is a quasi-atomic complement, i.e.\ for all $g\in G$, there is a strictly positive $n\in\bbn$ such that $ng\in H^{\prime}\oplus Q(G)$.

Notice that if $g\in H^\prime \subseteq  H^{\prime}\oplus Q(G)$, then our claim is established. If not, then the maximality of $H^{\prime}$ implies $\langle g, H^{\prime}\rangle$ has nontrivial intersection with $Q(G)$. In particular, for some $n>0$, $h\in H^{\prime}$, and $\alpha \in Q(G)$, we have that $ng+h=\alpha$. Clearly then $ng\in H^{\prime}\oplus Q(G)$. This establishes that $H^{\prime}$ is a quasi-atomic complement.
\end{proof}

We remark that a maximal quasi-atomic complement to $Q(G)$, which we have called $H$, is not unique, even with respect to order considerations. If we consider $\bbz^2$ under the lexicographic ordering, $Q(G)$ is uniquely determined (it is the subgroup $\mathbb{Z}\oplus 0$) but we have many choices for $H^{\prime}$. The subgroups generated by $(0,1)$ and $(1,1)$ are two distinct choices, for example. 

Notice that partially ordered abelian groups with nontrivial elements are necessarily torsion free. Indeed, for an element $x$ of finite order, say $n$, then we have the ordering $x \leq x^2 \leq x^3 \leq \cdots \leq x^{n-1} \leq x^n = e \leq x$. Antisymmetry in the partial order, $\leq$, insists that $x = e$. This leads us to the following theorem:

\begin{lem}\label{divisible_splitting}
Let $G$ be a directed po-group.  If there exists some quasi-atomic complement of $G$ that is divisible, say $H^{\prime} \subseteq G$, then $G = H^{\prime} \oplus Q(G)$.
\end{lem}
\begin{proof}Note that if $G = H^{\prime} \oplus Q(G)$, then the short exact sequence 
\[0 \to Q(G) \to G \to G/Q(G) \to 0\] necessarily splits, so $H^{\prime} = \frac{G}{Q(G)}$, and the direct sum $H^{\prime} \oplus \frac{G}{Q(G)}$ inherits the induced splitting partial order as previously described. Thus, we are primarily concerned with establishing this set equality.

We show the generators of $G$ are elements of $H^{\prime} \oplus Q(G)$. Since $G$ is directed, we only need concern ourselves with the positive elements. To this end, choose $g \in G$ to be positive. Since $H^{\prime}$ is a quasi-atomic complement, there exists some $n > 0$ such that $ng \in H^{\prime} \oplus Q(G)$. We write this as $ng = h + s$, for some $h \in H^{\prime}$, $s \in Q(G)$.

Since $H$ is divisible, we can write $h = nh_{0}$ for some $h_{0} \in H^{\prime}$.  Thus, $ng = nh_{0} + s$, and in particular, $n\left(g-h_{0}\right) + Q(G) = Q(G)$ as an element of $G/Q(G)$. Of course, $G/Q(G)$ is a po-group, and is hence torsion free, so we conclude $g - h_{0} \in Q(G)$, or rather $g = h_{0} + s_{0}$ for some $s_{0} \in Q(G)$.  Hence, $g \in H^{\prime} \oplus Q(G)$.
 \end{proof}

If any hypothesis of the theorem is violated, the corollary may not be true. For example, if $G$ is not generated by its positive elements, as in Example \ref{generatedbypositive}, the corollary fails.  Example \ref{last-ex} shows this theorem is immediately applicable to domains with pathological factorization behaviors since the quasi-atomic complement is divisible when certain domain elements admit $n^{th}$ roots for any $n \in \bbn$.

\begin{ex}\label{last-ex}
Let $K$ be a field and $x$ be any indeterminate. Consider the set of indeterminates $X = \left\{x^\alpha \mid \alpha \in \bbq^+\right\}$ be the collection of positive rational monomials, and consider the ring $R = K[X]_\mathfrak{m}$ where $\mathfrak{m}$ is the maximal monomial ideal. This ring has no irreducible elements, and so has no quasi-atomic elements. Thus we can trivially write $G = H^{\prime} \oplus 0$. Furthermore, $H^{\prime}$ is divisible. To demonstrate that $H^{\prime}$ is divisible, notice that multiplication in the ring is equivalent to the group operation on $G$. We may write any $g \in G$ as $g = x^\alpha U(R)$ and $ng = x^{n\alpha} U(R)$. Hence, for any $g \in G$, we have $g = x^{\alpha}U(R) = (x^{\alpha/n})^n U(R) = ng_0$ where $g_0 = x^{\alpha/n}U(R)$.
$\triangle$\end{ex}

\bibliography{biblio}{}

\begin{thebibliography}{10}

\bibitem{AAZ}
D.D. Anderson, D.F. Anderson, and M.~Zafrullah.
\newblock Factorization in integral domains.
\newblock {\em J. Pure Appl. Algebra}, 69:1--19, 1990.

\bibitem{anderson1990weakly}
DD~Anderson and Muhammad Zafrullah.
\newblock Weakly factorial domains and groups of divisibility.
\newblock {\em Proceedings of the American Mathematical Society},
  109(4):907--913, 1990.

\bibitem{boynton2013graph}
Jason~Greene Boynton and Jim Coykendall.
\newblock On the graph of divisibility of an integral domain.
\newblock {\em arXiv preprint arXiv:1304.0063}, 2013.

\bibitem{coykendall1999integral}
J.~Coykendall, D.E. Dobbs, and B.~Mullins.
\newblock On integral domains with no atoms.
\newblock {\em Communications in Algebra}, 27(12):5813--5831, 1999.

\bibitem{coy-emb}
J.~Coykendall and B.~Mammenga.
\newblock An embedding theorem.
\newblock {\em Journal of Algebra}, 325:177--185, 2011.

\bibitem{coy-ap-dom}
J.~Coykendall and M.~Zafrullah.
\newblock {AP}-domains and unique factorization.
\newblock {\em Journal of Pure and Applied Algebra}, 189:27--35, 2004.

\bibitem{fuchs2011partially}
L.~Fuchs.
\newblock {\em Partially ordered algebraic systems}.
\newblock Dover Publications, 2011.

\bibitem{gilmer1972multiplicative}
Robert~W Gilmer.
\newblock {\em Multiplicative ideal theory}, volume~12.
\newblock M. Dekker, 1972.

\bibitem{jensen1963characterizations}
CU~Jensen.
\newblock \protect{On characterizations of Pr\"{u}fer rings}.
\newblock {\em Math. Scand}, 13:90--98, 1963.

\bibitem{kaplansky1970commutative}
Irving Kaplansky.
\newblock {\em Commutative rings}.
\newblock Allyn and Bacon, 1970.

\bibitem{krull1932allgemeine}
W.~Krull.
\newblock Allgemeine bewertungstheorie.
\newblock {\em J. Reine Angew. Math}, 167(160-196):294, 1932.

\bibitem{lebowitz2016classifying}
Noah Lebowitz-Lockard.
\newblock Classifying subatomic domains.
\newblock {\em arXiv preprint arXiv:1610.05874}, 2016.

\bibitem{mott1974convex}
J.L. Mott.
\newblock Convex directed subgroups of a group of divisibility.
\newblock {\em Canad. J. Math}, 26(3):532--542, 1974.

\bibitem{mott1976exact}
Joe~L Mott and Michel Schexnayder.
\newblock Exact sequences of semi-value groups.
\newblock {\em Journal f{\"u}r Mathematik. Band}, 283(284):50, 1976.

\bibitem{mockor1983groups}
J.~Mo\v{c}ko\v{r}.
\newblock {\em Groups of divisibility}, volume~3.
\newblock Springer, 1983.

\bibitem{Roi-Poly}
M.~Roitman.
\newblock Polynomial extensions of atomic domains.
\newblock {\em Journal of Pure and Applied Algebra}, 87:187--199, 1993.

\bibitem{samuel1948ultrafilters}
Pierre Samuel.
\newblock Ultrafilters and compactification of uniform spaces.
\newblock {\em Transactions of the American Mathematical Society},
  64(1):100--132, 1948.

\bibitem{sheldon1973two}
P.B. Sheldon.
\newblock Two counterexamples involving complete integral closure in
  finite-dimensional pr\"{u}fer domains.
\newblock {\em Journal of Algebra}, 27(3):462--474, 1973.

\bibitem{zaks1976half}
Abraham Zaks.
\newblock Half factorial domains1.
\newblock {\em AMERICAN MATHEMATICAL SOCIETY}, 82(5), 1976.

\end{thebibliography}
\bibliographystyle{plain}

\end{document}